\documentclass[11pt,a4paper,reqno,oneside]{amsart}

\usepackage[left=0.75in,right=0.75in,bottom=0.75in,top=0.75in]{geometry}

\usepackage{tikz,amssymb,dsfont}
\usepackage{hyperref}
\usepackage{bm}
\usepackage{subcaption}
\usepackage{graphicx}

\DeclareMathOperator{\spec}{spec}
\DeclareMathOperator{\supp}{supp}

\DeclareMathOperator{\R}{\mathbb{R}}
\DeclareMathOperator{\N}{\mathbb{N}}

\DeclareMathOperator{\spann}{span}

\newcommand{\spece}{\spec_\mathrm{ess}}

\newcommand{\ddd}{\,\mathrm{d}}

\def \rr {{\mathbb R}}

\numberwithin{equation}{section}
\numberwithin{figure}{section}
\theoremstyle{definition}

\newtheorem{definition}{Definition}[section]
\newtheorem{remark}[definition]{Remark}
\newtheorem{notation}[definition]{Notation}

\theoremstyle{plain}

\newtheorem{theorem}[definition]{Theorem}
\newtheorem{proposition}[definition]{Proposition}
\newtheorem{lemma}[definition]{Lemma}
\newtheorem{corollary}[definition]{Corollary}

\makeatletter

\usepackage[
backend=biber,
sorting=nyt
]{biblatex}
\bibliography{lit}

\title{Eigenvalues of the magnetic Neumann Laplacian on domains with peaks}
\author{noah koerner }

\begin{document}

\begin{abstract}
	We consider the magnetic Neumann Laplacian on bounded domains in $\mathbb{R}^2$ with outward peaks. The operator is associated with a large magnetic field depending on a parameter $\lambda$, and we investigate the behaviour of the eigenvalues as $\lambda$ tends to $+\infty$. We show that their asymptotic expansion is influenced by the sharpness $q$ and geometry of the peak, and that its main term is of order $\lambda^\frac{2}{q+1}$.
    This is an extension of previous works on magnetic Neumann Laplacians in smooth domains and domains with corners.
\end{abstract}
\maketitle


\section{Introduction}
\subsection{Motivation}

This work is devoted to the spectral analysis of the magnetic Neumann Laplacian on bounded domains with outward peaks.

Let $\Omega \subseteq \R^2 $ be an open, bounded domain and $A = (A_1,A_2)$ a smooth magnetic potential associated with its magnetic field $B=\text{curl} \ A$ on $\Omega$. Then one defines $N_\lambda^\Omega$, the magnetic Laplacian on $\Omega$ with Neumann boundary conditions at $\partial \Omega$, as the unique self-adjoint operator in $L^2(\Omega)$ generated by the quadratic form
\[
n_\lambda^\Omega(f) = \int\limits_\Omega |(\nabla- i\lambda A)f|^2 \ddd x, \quad D(n_\lambda^\Omega) = H^1(\Omega).
\]

This operator appears when studying type II superconductivity as part of the linearization of the Ginzburg-Landau energy functional at the normal state. As one wants to analyse whether the normal state is a local minimum of the functional, one needs to examine the positivity of the linearization. It is therefore a natural consequence that there is much interest in the bottom of the spectrum of $N_\lambda^\Omega$ for large magnetic fields. For further details, we refer to \cite{Fournais_2010}. \cite{Raymond_2017} also gives a good introduction and foundation of the topic. 

If one considers bounded $\Omega$ with $C^0$ class boundary, $N_\lambda^\Omega$ has compact resolvent (see \cite[Chapter 5, Theorem 4.17]{Evans_2018}) and therefore infinitely many discrete eigenvalues tending to $+\infty$. We are interested in the behaviour of those eigenvalues for constant magnetic field $B=1$ and as $\lambda \to + \infty$.

The case of domains with smooth boundary has been explored by Bernhoff and Sternberg in \cite{Bernhoff_1998}, Helffer and Mohammed/Morame in \cite{Helffer_1996},\cite{Helffer_2001Bottles} as well as by Lu and Pan in \cite{Lu_Pan_1999}, \cite{LuPan_1999}. It was shown that the eigenstates are localised near points of the boundary where the curvature is maximal, and that the asymptotic expansion of the first eigenvalue is of the form
\[
E_1(N_\lambda^\Omega) = \Theta_0\lambda + \mathcal{O}(\lambda^{\frac{1}{2}}) \quad \text{ as } \lambda \to + \infty,
\]
where $\Theta_0$ is the de Gennes constant with $\Theta_0= \inf_{\xi \in \R} E_1(D_\xi) \in (0,1)$, and $D_\xi$ is the harmonic oscillator on the half-axis with Neumann boundary conditions at $0$.
Fournais and Helffer even established in \cite{Forunais_2006} a complete asymptotic expansion of the first few eigenvalues on domains whose boundary curvature has exactly one maximum and further proved its dependence on the curvature:

\[
E_1(N_\lambda^\Omega) = \Theta_0 \lambda - \kappa_{max}C_1 \lambda^{\frac{1}{2} } + \mathcal{O}(\lambda^{\frac{1}{4}}) \quad \text{ as } \lambda \to + \infty. 
\]

Here, $\kappa_{max}$ is the maximum of the curvature of $\partial\Omega$ and $C_1$ is a constant independent of the geometry of $\Omega$.

There have also been extensions to domains with piecewise smooth boundaries, namely specific curvilinear polygons in \cite{Bonnaillie_2006} by Bonnaillie-No\"el and Dauge.
Related to the study of curvilinear polygons is a certain model operator for the corners of the curvilinear polygon, the magnetic Neumann Laplacian on an infinite sector of opening $\alpha$, $ S_\alpha:= \{(x,y) \in \R^2: |\arg(x+iy)| < \alpha/2\}$ that we denote by $N_1^{S_{\alpha}}$.  Bonnaillie-No\"el showed in \cite{Bonaillie_2002} that due to the unboundedness of $S_\alpha$, $N_1^{S_\alpha}$ has essential spectrum that starts at $\Theta_0$ and that eigenvalues appear only for sufficiently small angles $\alpha$. In fact, there has been a recent preprint \cite{Kachmar_2026} by Kachmar and Sundqvist that showed the existence of at least one eigenvalue for all angles $\alpha \in (0,\pi)$. So assuming that the curvilinear polygon with $M$ corners with angles $\alpha_j \in (0,2\pi)\setminus \{\pi\}$ has at least one convex corner, one has $\mathcal{E}:= \min\limits_{i =1 ,\dots,M} E_1(N_1^{S_{\alpha_i}}) < \Theta_0 $ and the asymptotic expansion of the first eigenvalue of the magnetic Neumann Laplacian on the curvilinear polygon is of the form

\[
E_1(N_\lambda^\Omega) =  \mathcal{E}\lambda + \mathcal{O}(\lambda^{\frac{1}{2}})  \quad \text{ as } \lambda \to + \infty.
\]
Furthermore, the respective eigenstate is localised near the corner(s) corresponding to $\mathcal{E}$. 

Bonaillie-No\"el and Dauge have conjectured in \cite{Bonnaillie_2006} that $E_1(N_1^{S_\alpha})$ is strictly increasing on $(0,\pi)$ and does not exist on $[\pi,2\pi)$, and under the assumption that this holds, the constant $\mathcal{E}$ would be induced by the corner(s) with the smallest angle.
This would further support the trend that is already apparent through other results on this topic:
The most singular point of the boundary induces the first eigenvalue of $N_\lambda^\Omega$. Furthermore, the more "singular" this point is, the lower the eigenvalue gets.
Our findings on domains with outward peaks align with this pattern as well. Note that while such domains are also piecewise smooth,
they are not contained in the class of curvilinear polygons, as the border of $\Omega$ meets at the peak at an angle of zero. In fact, the eigenvalue asymptotics that emerges is very different from the one present in curvilinear polygons.

\subsection{Main result}

We are studying the operator $N_\lambda^\Omega$ as defined above, where $\Omega $ is a curvilinear polygon with an outward peak at the origin. To rigorously describe such $\Omega$, let us first introduce the peak $V_b$ of sharpness order $q>1$ which is defined as
    \[ 
    V_b := \{ (x,y) \in \R^2 \mid  x \in (0,b),\, a_-x^q < y <   a_+x^q  \}
    \]
    for some $b>0$ and $a_-,a_+ \in \R$ with $a_-<a_+$. A visualisation of $V_b$ can be found on Figure \ref{fig:peak and omega} (A).

We then call $\Omega$ a curvilinear polygon with an outward peak of sharpness $q>1$ at the origin, if there exists a decomposition $\partial \Omega = \Gamma_1 \cup \dots \cup \Gamma_M$ of $\partial \Omega$ into 
smooth curves with regular ends such that $\Gamma_{j-1}$ and $\Gamma_{j}$ meet at an angle of $\alpha_j \in (0,2\pi)\setminus\{\pi \}$ for all $j \in \{2, \dots, M\}$, while $\Gamma_M$ and $\Gamma_1$ meet at zero angle in the origin, and the singularity there is exactly of the form such that 

    \[
    \Omega \cap (-b,b)^2 = V_{b}
    \]
holds for a peak $V_b$ of sharpness order $q$. A visualisation of this kind of domain can be found on Figure \ref{fig:peak and omega} (B).

\begin{figure}[t]
    \centering
    \begin{subfigure}[b]{0.45\textwidth}
    
\tikzset{every picture/.style={line width=0.75pt}} 

\scalebox{0.8}{
\begin{tikzpicture}[x=0.75pt,y=0.75pt,yscale=-1,xscale=1]

\draw    (449.6,73.62) -- (450,226.02) ;
\draw  [draw opacity=0][fill={rgb, 255:red, 155; green, 155; blue, 155 }  ,fill opacity=0.41 ] (449.6,73.62) -- (450,226.02) -- (220,176.02) -- cycle ;
\draw [fill={rgb, 255:red, 255; green, 255; blue, 255 }  ,fill opacity=1 ]   (220,176.02) .. controls (304,166.42) and (406.8,101.62) .. (449.6,73.62) ;
\draw [fill={rgb, 255:red, 255; green, 255; blue, 255 }  ,fill opacity=1 ]   (220.4,176.02) .. controls (335.2,183.62) and (412.8,218.82) .. (450,226.02) ;
\draw    (190,176.02) -- (498,176.02) ;
\draw [shift={(500,176.02)}, rotate = 180] [fill={rgb, 255:red, 0; green, 0; blue, 0 }  ][line width=0.08]  [draw opacity=0] (8.4,-2.1) -- (0,0) -- (8.4,2.1) -- cycle    ;
\draw    (220,236.02) -- (220,68.02) ;
\draw [shift={(220,66.02)}, rotate = 90] [fill={rgb, 255:red, 0; green, 0; blue, 0 }  ][line width=0.08]  [draw opacity=0] (8.4,-2.1) -- (0,0) -- (8.4,2.1) -- cycle    ;

\draw (209,178.02) node [anchor=north west][inner sep=0.75pt]  [font=\small] [align=left] {$\displaystyle 0$};
\draw (371,138.02) node [anchor=north west][inner sep=0.75pt]   [align=left] {$\displaystyle V_{( 0,b)}$};
\draw (439,178.02) node [anchor=north west][inner sep=0.75pt]  [font=\small] [align=left] {$\displaystyle b$};

\end{tikzpicture}}
    \caption{The peak $V_{b}$}
    \end{subfigure}
    \begin{subfigure}[b]{0.45\textwidth}
    \scalebox{0.8}{
\begin{tikzpicture}[x=0.75pt,y=0.75pt,yscale=-1,xscale=1]

\draw [fill={rgb, 255:red, 155; green, 155; blue, 155 }  ,fill opacity=0.5 ]   (201.44,124.22) .. controls (242.67,95.5) and (299.87,57.47) .. (313.87,77.07) .. controls (327.87,96.67) and (385.87,48.27) .. (421.07,83.87) ;
\draw [fill={rgb, 255:red, 155; green, 155; blue, 155 }  ,fill opacity=0.5 ]   (201.6,183.87) .. controls (254,196.17) and (301.87,254.73) .. (341.87,224.73) ;
\draw [fill={rgb, 255:red, 155; green, 155; blue, 155 }  ,fill opacity=0.5 ]   (341.87,224.73) .. controls (348.13,186.27) and (383.47,195.07) .. (411.33,205.47) .. controls (439.2,215.87) and (490.8,185.07) .. (468.93,155.07) ;
\draw [fill={rgb, 255:red, 155; green, 155; blue, 155 }  ,fill opacity=0.5 ]   (421.07,83.87) .. controls (428.83,78.05) and (445.09,70.6) .. (452.61,82.84) .. controls (460.13,95.07) and (507.73,130.27) .. (468.93,155.07) ;
\draw  [draw opacity=0][fill={rgb, 255:red, 155; green, 155; blue, 155 }  ,fill opacity=0.5 ] (201.44,124.22) -- (201.6,183.87) -- (108.27,164.3) -- (158.94,142.5) -- cycle ;
\draw [fill={rgb, 255:red, 255; green, 255; blue, 255 }  ,fill opacity=1 ]   (108.43,164.3) .. controls (142.52,160.54) and (184.23,135.18) .. (201.6,124.22) ;
\draw [fill={rgb, 255:red, 255; green, 255; blue, 255 }  ,fill opacity=1 ]   (108.43,164.3) .. controls (155.01,167.27) and (186.5,181.05) .. (201.6,183.87) ;
\draw  [draw opacity=0][fill={rgb, 255:red, 155; green, 155; blue, 155 }  ,fill opacity=0.5 ] (468.93,155.07) -- (341.87,224.73) -- (201.6,183.87) -- (201.44,124.22) -- (421.07,83.87) -- cycle ;
\draw [color={rgb, 255:red, 0; green, 0; blue, 0 }  ,draw opacity=1 ][fill={rgb, 255:red, 255; green, 255; blue, 255 }  ,fill opacity=1 ]   (341.87,224.73) .. controls (347.14,192) and (373.14,195.14) .. (390,198.86) ;

\draw (303.07,134.33) node [anchor=north west][inner sep=0.75pt]   [align=left] {$\displaystyle \Omega $};

\end{tikzpicture}}

    \caption{A curvilinear polygon with an outward peak}
    \end{subfigure}
    
    \caption{Visualization of a peak and a curvilinear polygon with an outward peak}
    \label{fig:peak and omega}
\end{figure}
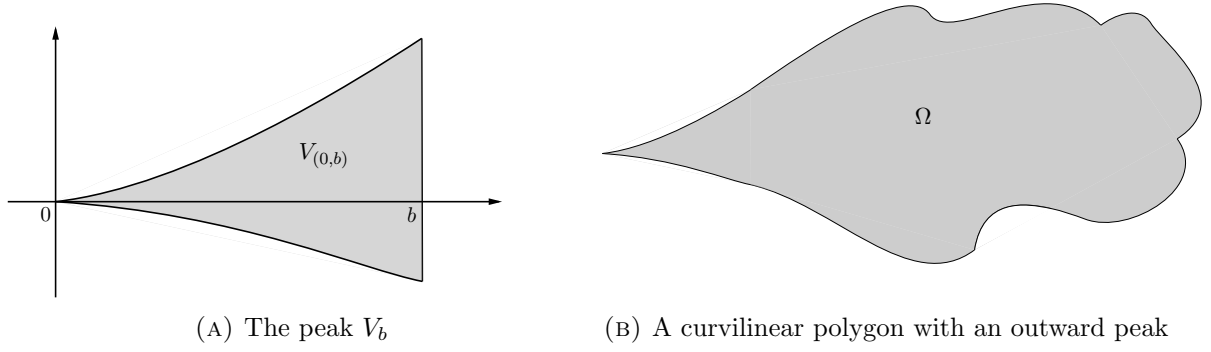

We summarise our main result in the following theorem:

\begin{theorem}
\label{main theo}
    Let $\Omega \subseteq \R^2 $ be a curvilinear polygon with an outward peak at the origin of sharpness order $q$, $n\in \N$, and $\kappa = \frac{1}{q+1} \frac{2q+3}{2+3q} $. Then, as $\lambda \to +\infty$, the $n$th eigenvalue of $N_\lambda^\Omega$ has the following asymptotics:
    \[
    E_n(N_\lambda^\Omega) = \lambda^{\frac{2}{q+1}} ({a_+-a_-})^{\frac{2}{q+1}} E_n(T) + \mathcal{O}(\lambda^{2\kappa}),
    \]
    where the operator $T$ is defined as the Friedrichs' extension in $L^2(0,\infty)$ of 
    \[
    C^\infty_c(0,\infty) \ni h \mapsto- h'' + h\left( \frac{1}{12} s^{2q}  + \frac{q(q-2)}{4s^2} \right).
    \]
\end{theorem}

Note that the order of the first term of the expansion is now ${\frac{2}{q+1}}$ and no longer $1$, as was the case for smooth domains and curvilinear polygons. This means that for domains with peaks, the first eigenvalue of $N_\lambda^\Omega$ diverges to $+\infty$ much more slowly than it is the case for all previously examined domains. In fact, the sharper the peak is, the slower the approach, which aligns with the already mentioned pattern of strong singularities inducing low eigenvalues.

The structure of this paper goes as follows:
We start in Section \ref{sec: Prelim} by laying down some notation and stating useful lemmas concerning the comparison of operators. We also introduce the $1$-dimensional model operator $T_\lambda$ that is needed for the spectral analysis of the peak $V_{b}$ and show some of its attributes as well as recall some basic properties of the magnetic Laplacian in $2$ dimensions with Dirichlet/Neumann boundary conditions.

The main portion of this paper, Section \ref{sec: peak}, is devoted to the spectral analysis of the magnetic Laplacian on sections of the peak $V_{b}$ with mixed Dirichlet and Neumann boundary conditions. We start by providing a general upper bound, then proceed by showing lower bounds for three kinds of sections of $V_{b}$: Directly at the peak, near the peak, and away from the peak. These bounds are established with the help of the $1$-dimensional model operator $T_\lambda$.

In Section \ref{sec: Endzeit}, we collect the results of Section \ref{sec: peak} to prove the main Theorem \ref{main theo}.

\section{Preliminaries}
\label{sec: Prelim}
\subsection{Notation}

A sesquilinear form $q: D(q) \times D(q) \to \mathbb{C}$ will always be referred to by a lower case letter and we will write $q(u) = q(u,u)$ for any $u \in D(q) \subseteq \mathcal H$ where $\mathcal{H}$ is an infinite-dimensional Hilbert space. The linear operator that is generated by a closed, symmetric sesquilinear form $q$ in  $\mathcal{H}$ will be denoted by the corresponding uppercase letter $Q$, i.e., one has 
\[
 q(u,v) = \langle u, Qv \rangle_\mathcal{H} \quad \text{ for all } u \in D(q)\text{ and } v \in D(Q).
\]

Using Min-Max principle, we denote the $n$th Rayleigh quotient $\Lambda_n(Q)$ as
\[
\Lambda_n(Q) = \inf\limits_{ \underset{\dim U =n}{U \subseteq D(q)}} \sup\limits_{ \underset{u \neq 0}{u \in U}} \frac{q(u,u)}{\|u \|^2_{\mathcal{H}}}
\]
and the $n$th eigenvalue of $Q$ as $E_n(Q)$ (if it exists).

\subsection{IMS partition and comparing operators}
We recall a standard lemma concerning the comparison of operators that follows directly from the Min-Max principle.
\begin{lemma}
\label{comparing operators}
    Let $Q, \tilde Q$ be two lower semibounded self-adjoint operators 
    in Hilbert spaces $\mathcal{H}, \tilde{\mathcal{H}}$. Let $\tilde D \subset D(\tilde q)$ be a dense subspace and $J: \tilde D \rightarrow D(q) $ such that it holds
    \[ \| J u \|^2_\mathcal{H} = \|u\|^2_{\tilde{\mathcal{H}}}, \quad q(Ju) \leq \tilde q (u) \quad \forall \ u \in \tilde D .
    \]
    Then, for any $n \in \N$, one has $  \Lambda_n(Q) \leq \Lambda_n(\tilde Q)$.

\end{lemma}

We will frequently need to employ the IMS localization formula. Its formulation in this specific setting goes as follows:

\begin{lemma}
\label{lem: IMS partition}
    Let $U \subseteq \R^2$ be an open set, $A$ a smooth magnetic potential on $U$ such that curl $A =1$ holds,
    \[
    q_\lambda^U(f) := \int\limits_U |(\nabla - i\lambda A)f|^2 \ddd x, \quad D(q_\lambda^U) \subseteq H^1(U)
    \]
    a closed, lower semi-bounded quadratic form and 
    $\chi_j \in C^\infty(\overline U)$, $0\leq \chi_j \leq 1$, $\chi_1^2+ \dots+ \chi_n^2 =1$ then it holds

    \begin{align*}
    &q^U_\lambda(f) = \sum\limits_{j=1}^M q_\lambda^U(\chi_j f) -\sum\limits_{j=1}^M \| \nabla \chi_j f \|^2_{L^2(U)} , \\
    &q^U_\lambda(f) + \|f\|^2_{L^2(U)} \sum\limits_{j=1}^n \| \nabla \chi_j \|^2_{\infty} \geq \sum\limits_{j=1}^n q_\lambda^U(\chi_j f).
    \end{align*}
\end{lemma}
\begin{proof}
    Analogue to \cite[Lemma 4.4]{Bonnaillie_2006}. 
\end{proof}

\subsection{1-dimensional model operators}
\label{sec: 1dim model op}

For the analysis of the magnetic Laplacian on the peak $V_{b}$, we need to introduce the following $1$-dimensional model operator as well as a cropped version of it. We will examine some of their properties and how they are related to each other.

\begin{definition}
    For $\lambda >0$, denote by $T_\lambda$ the symmetric differential operator given by the Friedrichs extension in $L^2(0,\infty)$ of 
    \[
        C^\infty_c(0, \infty) \ni h \mapsto -h'' + \left( \frac{(a_+-a_-)^2}{12} \lambda^2s^{2q} + \frac{q(q-2)}{4s^2} \right)h.
    \]
\end{definition}

We summarise the properties of $T_\lambda$ in the following proposition.

\begin{proposition}
\label{lem: properties model operator}
    $T_\lambda$ is positive, unitary equivalent to $ \lambda^{\frac{2}{q+1}} T_1 $, and its spectrum consists of simple, discrete eigenvalues.
\end{proposition}

\begin{proof}
The positivity follows directly from the one-dimensional Hardy inequality
\begin{equation}
\label{eq: Hardy}
\int\limits_0^\infty |h'|^2\ddd s \geq \int\limits_0^\infty \frac{|h|^2}{4s^2} \ddd s
\end{equation}
that holds for any $h \in C^\infty_c(0,\infty)$. We write
    \begin{align*}
	t_\lambda (h) &= \int\limits_0^\infty |h'|^2 + |h|^2\left(  \frac{(a_+-a_-)^2}{12} \lambda^2 s^{2q}+\frac{q(q-2)}{4s^2} \right) \ddd s \\
    &\geq  \int\limits_0^\infty |h|^2\left( \frac{(a_+-a_-)^2}{12}\lambda^2 s^{2q} + \frac{(q-1)^2}{4s^2} \right) \ddd s
    \geq \int\limits_0^1 |h|^2  \frac{(q-1)^2}{4s^2}  \ddd s + \int\limits_1^\infty |h|^2 \frac{(a_+-a_-)^2}{12}\lambda^2 s^{2q} \ddd s \\
    &\geq \frac{(q-1)^2}{4}\int\limits_0^1 |h|^2 \ddd s + \frac{(a_+-a_-)^2}{12}\lambda^2 \int\limits_1^\infty |h|^2 \ddd s 
    \geq \min \left\{ \frac{(a_+-a_-)^2}{12}\lambda^2 ,\frac{(q-1)^2}{4} \right\} \|h\| ^2_{L^2(0,\infty)},
	\end{align*}
    and conclude, noting that $C^\infty_c$ is dense in the domain of $t_\lambda$ by construction. 
    
    Furthermore, the unitary equivalence can be shown by doing a rescaling $ h(s) \mapsto \sqrt{\alpha} h(\alpha s)$ which gives $ T_\lambda \cong \alpha^2T_{\lambda \alpha^{-q-1}} $, and setting $\alpha = \lambda^{\frac{1}{q+1}}$.
    
\end{proof}
The other properties are shown in Lemma \ref{lem: ess spec empty} and Lemma \ref{lem: EV simple}, but to simplify notation we first introduce a truncated version of $T_\lambda$.
\begin{definition}
    Let $I\subseteq (0,\infty)$ be an interval and set $T_\lambda^I $ as the Friedrichs extension in $L^2(I)$ of
    \[
        C_c^\infty(I) \ni h \mapsto T_\lambda h.
    \]
\end{definition}

\begin{remark}
\label{remark}

	Let $I\subseteq (0,\infty)$ be a bounded interval; then it can be shown that $D(t_\lambda^I) = H^1_0(I)$ holds as the norm induced by $ \langle \cdot,\cdot\rangle_{t_\lambda^I} := \langle \cdot,\cdot \rangle_{L^2(0,I)} + t^I_\lambda(\cdot,\cdot)$ is equivalent to the Sobolev norm on $I$.

	Namely, for any $h \in C^\infty_c(I)$ and any $\delta \in (0, \min\{1,(q-1)^2\})$ one has, on one hand, due to the one-dimensional Hardy inequality (\ref{eq: Hardy}):
    \begin{align*}
    t_\lambda^I(h) &=\int\limits_I |h'|^2 (1-\delta + \delta) + |h|^2\left(\frac{(a_+-a_-)^2}{12}\lambda^2 s^{2q} + \frac{q(q-2)}{4s^2}\right) \ddd s \\
    &\geq \int\limits_I \delta |h'|^2 + |h|^2 \frac{(q-1)^2-\delta}{4s^2} \ddd s
    \geq 
    \int\limits_I\delta |h'|^2 \ddd s.
    \end{align*}
    It follows from the density of $C^\infty_c$ that $ \| h \|_{t_\lambda^I} \geq \sqrt{\delta} \| h\|_{H^1(I)} $ holds. On the other hand, we have for $q \in (1,2)$:
    \[
    t_\lambda^I(h) \leq \int\limits_I |h'|^2 + |h|^2 \frac{(a_+-a_-)^2}{12}\lambda^2 (\sup I)^{2q} \ddd s \quad 
    \Rightarrow \quad \| h\|_{t_\lambda^I} \leq \sqrt{1+\frac{(a_+-a_-)^2}{12} \lambda^2 (\sup I)^{2q}} \|h\|_{H^1(I)},
    \]
    and if $q \geq 2$ holds, we can employ (\ref{eq: Hardy}):
    \[
    t_\lambda^I(h) \leq \int\limits_I \Bigl[1+q(q-2)\Bigr] |h'|^2 + |h|^2 \frac{(a_+-a_-)^2}{12} \lambda^2 (\sup I)^{2q} \ddd s .\]
    Thus,
    \[  
    \| h\|_{t_\lambda^I} \leq \sqrt{(q-1)^2 +\frac{(a_+-a_-)^2}{12} \lambda^2 (\sup I)^{2q}} \|h\|_{H^1(I)}.
    \]
    So in both cases, there exists a constant $c>0$ such that 
    $ \delta \|h\|_{H^1(I)} \leq \| h\|_{t_\lambda^I} \leq c\|h \|_{H^1(I)}$ holds.
\end{remark}

We now finish the proof of Proposition \ref{lem: properties model operator} with the next two Lemmas.

\begin{lemma}
\label{lem: ess spec empty}
	The essential spectrum of $T_\lambda$ is empty. 
\end{lemma}
\begin{proof}
	 We do an IMS partition. Let
	\[
	\chi, \tilde\chi \in C^\infty , \quad 0\leq \chi,\tilde \chi \leq 1, \quad \chi^2 + \tilde\chi^2 \equiv 1, \quad
	\chi(x) = \begin{cases}
		1, & x \leq \frac{1}{2} \\
		0, & x \geq 1
	\end{cases}   ,
	\]
	and set $\chi_R(x) = \chi(x/R)$, $\tilde \chi _R (x) = \tilde \chi (x/R)$ then since we have 
	$\supp \chi_R \subseteq [0,R]$ and $\supp \tilde\chi_R \subseteq [R/2,\infty)$ the operator  
    $J: C^\infty_c(0,\infty) \to C^\infty_c(0,R) \oplus C^\infty_c(R/2,\infty) , \ h \mapsto (\chi_R h, \tilde \chi_R h)$
    is well-defined, fulfills $\|Jh\|_{L^2} = \|h\|_{L^2}$ for all $h \in C^\infty_c(0,\infty)$, and one has due to Lemma \ref{lem: IMS partition} for some $c_0 >0$
    \[
    t_\lambda (h)  
    = t_\lambda^{(0,R)}(\chi_r h) + t_\lambda^{(R/2, \infty)}(\tilde \chi_R h) - \frac{c_0}{R ^2} \|h \|^2_{L^2(0,\infty)} 
    = \left(t_\lambda^{(0,R)} \oplus t_\lambda^{(R/2,\infty)} \right) (Jh) - \frac{c_0}{R ^2} \|h \|^2_{L^2(0,\infty)} .
    \]
    Lemma \ref{comparing operators} then implies
	\begin{equation}
    \label{eq: end me}
	\Lambda_n(T_\lambda) + \frac{c_0}{R^2} \geq \Lambda_n(T_\lambda^{(0,R)} \oplus T_\lambda^{(R/2,\infty)}).
	\end{equation}
	Note, that due to Remark \ref{remark} one has $D(t_\lambda^{(0,R)}) = H^1_0(0,R)$ for any $R>0$ and since the embedding $H^1_0(0,R) \hookrightarrow L^2(0,R) $ is compact, $T_\lambda^{(0,R)}$ has compact resolvent and therefore no essential spectrum.
	
	On the other hand, we can employ for any $h\in C^\infty_c(R/2,\infty)$ the Hardy inequality  (\ref{eq: Hardy}):
	\[
	t_\lambda^{(R/2,\infty)} (h) \geq \int\limits^\infty_{R/2} |h|^2 \left(\frac{(a_+-a_-)^2 }{12} \lambda^2 s^{2q} + \frac{(q-1)^2}{4s^2} \right) \ddd s \geq \frac{(a_+-a_-)^2}{12} \lambda^22^{{-2q}} R^{2q} \|h\|^2_{L^2(R/2,\infty)}.
	\]
	Thus, $\inf \spece T_\lambda^{(R/2,\infty)} $ is bounded from below by $ c_1 R^{2q}$ for some $c_1>0$ and we conclude by noting that due to (\ref{eq: end me}) it holds
    \[
	\inf \spece T_\lambda = \lim\limits_{R \to + \infty} c_1R^{2q} = + \infty. \qedhere
	\]
\end{proof}

\begin{lemma}
\label{lem: EV simple}
	All eigenvalues of $T_\lambda$ are simple.
\end{lemma}
\begin{proof}
    Let $\psi$ be an eigenfunction of $T_\lambda$ with associated eigenvalue $E$ then $\psi$ solves the differential equation
    \begin{equation}
    \label{eq: Last one}
    -\psi'' + \psi \left( \frac{(a_+-a_-)^2}{12}\lambda^2 s^{2q} + \frac{q(q-2)}{4s^2}\right) = E \psi,
    \end{equation}
    which is square integrable on $(0,\infty)$ and therefore square integrable on $(1,\infty)$ as well. Note that there exists $c \in \R$ such that
    \begin{equation*}
    \left( \frac{(a_+-a_-)^2}{12}\lambda^2 s^{2q} + \frac{q(q-2)}{4s^2}\right) \geq c
    \end{equation*}
    holds for all $s\in(1,\infty)$. This implies that we have the limit point case at infinity (see \cite[Theorem 6.6]{Weidman_1987}), and it follows that for any $E\in \R$ the space of solutions to (\ref{eq: Last one}) that are integrable on $(1,\infty)$ is at most one-dimensional (\cite[Theorem 5.6]{Weidman_1987}).
\end{proof}

We proceed with the analysis of the cropped operator $T_\lambda^I$ by showing an upper and a lower bound for $T_\lambda^I$ if $I$ is a bounded interval not containing $0$.

\begin{lemma}
\label{lem: model operator middle piece}
    Let $b_1,b_2 >0$ and $0 \leq \eta< \kappa< \frac{1}{q+1} $ and set  $I = (b_1\lambda^{-\kappa},b_2 \lambda^{-\eta})$ for sufficiently large $\lambda$.
    Then there exist $C_1,C_2 >0$ such that
    \[
    C_1\lambda^{2(1-q \kappa)} \leq E_1 (T_\lambda^I) \leq C_2 \lambda^{2(1-q \kappa)}
    \]
    holds as $\lambda \to +\infty$.
\end{lemma}

\begin{proof}
    For the first inequality, we write for any $h\in C^\infty_c(I)$ with the help of the Hardy inequality (\ref{eq: Hardy})
    \begin{align}
    \label{eq: lower bound tlambdaI}
    \begin{split}
    t_\lambda^I(h) &= \int\limits_I |h'|^2 + |h|^2\left(\frac{(a_+-a_-)^2}{12 } \lambda^2 s^{2q} + \frac{q(q-2)}{4 s^2}\right) \ddd s \\
    &\geq \int\limits_I |h|^2 \left( \frac{(a_+-a_-)^2}{12 } \lambda^2 s^{2q } + \frac{(q-1)^2}{4s^2} \right) \ddd s \\
    &\geq \frac{(a_+-a_-)^2}{12} \int\limits_{b_1 \lambda^{-\kappa}}^{b_2 \lambda^{-\eta}}|h|^2 \lambda^2 s^{2q} \ddd s 
    \geq \frac{(a_+-a_-)^2}{12} b_1^{{2q}} \lambda^{2(1-q \kappa)} \| h\|^2_{L^2(I)}
    \end{split}
    \end{align}
    and conclude using Min-Max. To show the second inequality, let $\psi \in C^\infty_c(\R)$ with $\supp \psi \subseteq [1,2]$ and $\| \psi\|_{L^2(\R)} =1$. 
    Through a rescaling, we define the function $\psi_\lambda := \psi(s\lambda^\kappa b_1^{-1}) \in H^1_0(I)$ that has norm $\| \psi_\lambda \|^2_{L^2(0,\infty)} = \lambda^{-\kappa}b_1 \| \psi\|^2_{L^2(\R)} = \lambda^{-\kappa}b_1$. We write  
    \[
    \int\limits_I |\partial_s(\psi_\lambda)|^2 \ddd s 
    = \int\limits_I b_1^{-2} \lambda^{2 \kappa} |\psi'(s\lambda^\kappa b_1^{-1})|^2 \ddd s 
    = b_1^{-2} \lambda^{2\kappa} (\lambda^{-\kappa}b_1) \| \psi' \|^2_{L^2(\R)}   = b_1^{-2}  \lambda^{2\kappa}  \|\psi' \|^2_{L^2(\R)} \| \psi_\lambda \|^2_{L^2(I)}.
    \]  
    By combining this with the inequality $ \kappa < 1-q\kappa $ that holds for any $\kappa \in (0,\frac{1}{q+1})$ we get for some $c_1>0$ 
    \begin{align*}
    t_\lambda^I (\psi_\lambda) 
    &\leq \int\limits_{b_1 \lambda^{- \kappa}}^{2b_1 \lambda^{-\kappa}} |\partial_s(\psi_\lambda) |^2 + |\psi_\lambda|^2 \left(   \frac{(a_+-a_-)^2}{12} \lambda^2 s^{2q} + \frac{|q(q-1)|}{4 s^2} \right) \ddd s \\
    &\leq \int\limits_{b_1 \lambda^{- \kappa}}^{2b_1  \lambda^{-\kappa}}  |\psi_\lambda|^2 \left[ \left(\frac{|q(q-1|)}{4b_1^2} +  b_1^{-2}   \|\psi' \|^2_{L^2(\R)}\right) \lambda^{2 \kappa} + \frac{(a_+-a_-)^2}{12} (2b_1)^{2q} \lambda^{2(1-q \kappa)}\right] \ddd s \\
    &\leq c_1 \lambda^{2(1-q\kappa)} \|\psi_\lambda \|^2_{L^2(I)}.
    \end{align*}
    The Min-Max principle then implies $E_1(T_\lambda^I) \leq c_1 \lambda^{2(1-q\kappa)}$ which concludes the proof. 
\end{proof}

Lastly, we compare $T_\lambda^I$ with $T_\lambda$ for a specific choice of $I$.

\begin{lemma}
\label{lem: model operator comparison cropped}
    Let $n\in \N$, $\kappa \in(0, \frac{1}{q+1}) $ and $I = (0,\lambda^{-\kappa})$  then there exists $C >0$ such that
    \begin{equation}
    \label{eq: inequality 1 dim}
        \lambda^{\frac{2}{q+1}} E_n(T_1) 
        \leq E_n(T_\lambda^{I}) 
        \leq \lambda^{\frac{2}{q+1}} \left(E_n(T_1) + C \lambda^{2(\kappa-\frac{1}{q+1})}\right) = \lambda^{\frac{2}{q+1}} E_n(T_1) + C \lambda^{2\kappa}
    \end{equation}
    holds as $\lambda \to +\infty$.
\end{lemma}
\begin{proof}
    For the first inequality of (\ref{eq: inequality 1 dim}) we set $ J_1:C^\infty_c(0, \lambda^{-\kappa}) \to C^\infty_c(0,\infty)$ as the continuation with $0$ onto $(0,\infty)$.
    As one has $ \| J_1h \|^2_{L^2(0,\infty)} = \|h\|^2_{L^2(I)} $ and $ t_\lambda(J_1h) = t^I_\lambda(h) $, Lemma
    \ref{comparing operators} and Proposition \ref{lem: properties model operator} give
    \[
    \lambda^{\frac{2}{q+1}}E_n(T_1) =E_n(T_\lambda) \leq E_n(T_\lambda^I).
    \]

    For the second inequality, let $\chi, \tilde\chi \in C^\infty(\R_+)$ such that $\chi^2 + \tilde\chi^2 =1$ and $\chi(x) = 0$ for $x > \frac{3}{4}$ and $\tilde\chi(x) = 0$ for $x < \frac{1}{2}$ holds. Set
    \[
    \chi_\lambda(s) := \chi(s \lambda^\kappa), \quad \tilde\chi_\lambda(s) := \tilde \chi (s \lambda^\kappa).
    \]
    We do an IMS partition. Set $c_0:= \|\chi'\|^2_\infty + \| \tilde\chi' \| ^2_\infty$ and $I_1:= (\lambda^{-\kappa}/4,\infty )$
    as well as
    \[J_2:C_c^\infty(0, \infty) \to C_c^\infty(I) \oplus C_c^\infty(I_1),  \quad J_2(h) = (\chi_\lambda h, \tilde \chi_\lambda h).
    \] 
    Note that one has $ \| h\|_{L^2} = \|J_2 h \|_{L^2}$ for any $h \in C^\infty_c$ and due to Lemma \ref{lem: IMS partition}
    \[
    t_\lambda(h) +c_0  \lambda^{2\kappa} \| h\|^2_{L^2(0,\infty)} \geq (t_\lambda^I \oplus t_\lambda^{I_1})(J_2h).
    \]
    Therefore, using Lemma \ref{comparing operators} we have
    \[
    E_n(T_\lambda) + c_0\lambda^{2 \kappa}\geq E_n(T_\lambda^{I} \oplus T_\lambda^{I_1}).
    \]
    Analogue to (\ref{eq: lower bound tlambdaI}) one can show that there exists $c_1 >0$ such that for sufficiently large $\lambda$ it holds
    \[
    T_\lambda^{I_1} \geq c_1 \lambda^{2(1-q\kappa)} .
    \]
    By combining this with $ E_n(T_\lambda) +c \lambda^{2\kappa} = \lambda^{\frac{2}{q+1}}(E_n(T_1) + c_0 \lambda^{2(\kappa-\frac{1}{q+1})})$ and noting that for any $\kappa \in (0,\frac{1}{q+1})$ one has $2(1-q\kappa) >2/(q+1) $ as well as $ 2(\kappa-1/(q+1))<0 $ one can deduce that, as $\lambda \to + \infty$, it holds
    \[
     \lambda^{\frac{2}{q+1}}E_n(T_1) + c_0 \lambda^{2\kappa}=E_n(T_\lambda) +c_0 \lambda^{2\kappa} \geq E_n(T_\lambda^{I}). \qedhere
    \]
    
\end{proof}

\subsection{2-dimensional model operators}
We will need some general results about the magnetic Laplacian in $\R^2$ and start by fixing a choice of the magnetic potential $A$ associated with a constant magnetic field $B=1$.

\begin{definition}
For an open, bounded domain $U \subseteq \R^2$ with $C^0$ class boundary, we set
    \begin{align*}
    &n^U_\lambda(f) = \int\limits_U | \partial_x f|^2 + |\partial_y f - i \lambda x f|^2 \ddd x \ddd y, \quad D(n^U_\lambda) = H^1(U), \\
    &d^U_\lambda(f) = \int\limits_U | \partial_x f|^2 + |\partial_y f - i \lambda x f|^2 \ddd x \ddd y, \quad D(d^U_\lambda) = H^1_0(U). 
    \end{align*}
	As usual, $N_\lambda^U$ and $D_\lambda^U$ refer to the operators generated by these quadratic forms, which are the magnetic Laplacians on $U$ with Neumann/Dirichlet boundary conditions at $\partial U$.  
\end{definition}

Let us first consider some basic properties of a slightly altered version of the magnetic Laplacian on the half-plane with Neumann boundary conditions.

\begin{lemma}
\label{lem: 2-dim model op}
Set $\R^2_+ := \{(x,y) \in \R^2 \mid  y>0\}$. Then 
$N_\lambda^{\R^2_+}$ is unitary equivalent to the operator generated by
\[
\tilde n_\lambda^{\R^2_+} (f) := \int\limits_{\R^2_+} |\partial_y f|^2 + |\partial_x f +i\lambda yf|^2 \ddd x \ddd y , \quad D(\tilde n_\lambda^{\rr^2_+}) = H^1(\R^2_+)
\]
and it holds
\[
N_\lambda^{\R^2_+} \geq \Theta_0 \lambda
\]
where $\Theta_0$ is the de-Gennes constant defined via
\[
\Theta_0 := \inf\limits_{\xi \in \R} E_1(D_\xi) \in (0,1).
\]
Here, $D_\xi = -\partial_t^2 + (t+\xi)^2$ is the harmonic oscillator on the half-axis with Neumann boundary conditions at $0$.
\end{lemma}
\begin{proof}
The unitary equivalence follows directly by doing a change of variables, and the other properties were shown, for example, in
    \cite[Chapters 3.2 and 4.3]{Fournais_2010}.
\end{proof}

Lastly, we establish a simple lower bound for magnetic Laplacians with Dirichlet boundary conditions.

\begin{lemma}
\label{lem: lower bound dirichlet}
Let $U \subseteq \R^2 $ be an open and bounded domain. Then
\[
E_1(D_\lambda^U) \geq \lambda
\]
\end{lemma}
\begin{proof}
    It is a well-known fact (see, for example, \cite[Chapter 4.2]{Fournais_2010}) that the magnetic Laplacian on $\R^2$ is bounded from below by $\lambda$. The lower bound for $D_\lambda^U$ follows through Dirichlet-monotony. 
\end{proof}

\section{Spectral asymptotics on the peak}
\label{sec: peak}

In this section, we analyse the spectral properties of the magnetic Laplacian on sections of a peak with mixed boundary conditions. The concrete geometric setting is established in the following definition.

\begin{definition}
    Let $q>1$, $a_-,a_+ \in \R$ with $a_-<a_+$ and $I\subseteq (0,\infty)$ an open interval. Set $a_{max} := \max\{ |a_+|,|a_-| \}$, then the set
    \[
    V_I := \{ (x,y) \in \R^2 \mid x \in I, a_-x^q <y<a_+x^q \}
    \]
    is called a section $I$ of an outward peak of sharpness order $q$.
    We also set the rectangle
    \[
    \Pi_I := I \times (a_-,a_+).
    \]
\end{definition}

The operator whose spectral properties are of interest is a magnetic Laplacian on $V_I$ with Dirichlet boundary conditions imposed on the straight parts of $\partial V_I$ and Neumann boundary conditions at the rest of the boundary:

\begin{definition}
    Let $U \subseteq \R^2$ be an open set and $I \subseteq \R$ an open
    interval and set
    \[
        H^1_I(U) := \{ f \in H^1(U) \mid \exists \, c_0, c_1 \in I, \, c_0 < c_1 : f(x,y) = 0 , \,\forall \, (x,y) \in U  \text{ with } x \notin [c_0,c_1] \}
    \]
    We define the quadratic form of the magnetic Laplacian on $V_I$ with mixed Dirichlet and Neumann boundary conditions as
    \[
        q_\lambda^{I} (f) = n_\lambda^{ V_I } ( f ), \quad D ( q^{I}_\lambda ) =  \text{the closure of } H^1_I ( V_I ) 
        \text{ in } H^1(V_I).
    \]
    As usual, $Q_\lambda^I$ refers to the self-adjoint operator in $L^2$ generated by $q_\lambda^I$.
\end{definition}

The general goal of this section is to show that the magnetic Laplacian on the peak $V_{(0,b)}$ for some fixed $b>0$ induces eigenvalues of order $ \lambda^{\frac{2}{q+1}}$. This is achieved by splitting up the peak into three domains via an IMS partition and analysing the eigenvalue asymptotics of the respective magnetic Laplacians separately. To be more precise, we define for some $\kappa > 1/(2q)$ and $b,b_1,b_2,b_3 >0$ the intervals
\[
I_1 := (0,\lambda^{-\kappa}), \quad  I_2 := (b_1 \lambda^{-\kappa}, b_2 \lambda^{-\frac{1}{2q}}), \quad I_3:= (b_3 \lambda^{-\frac{1}{2q}},b).
\]
Then the analysis of the spectral properties of $Q_\lambda^{I_j}$, $j=1,2,3$ yields that $E_n(Q_\lambda^{I_1}) \sim \lambda^{\frac{2}{q+1}}$ and that $E_1(Q_\lambda^{I_2})$ and $ E_1(Q_\lambda^{I_3}) $ diverge to $+\infty$ faster than $\lambda^{\frac{2}{q+1}}$ as $\lambda \to \infty$. This shows that the singularity at the peak induces eigenvalues of order $\lambda^{\frac{2}{q+1}}$ and that the rest of the peak has no significant impact on the eigenvalue asymptotics.

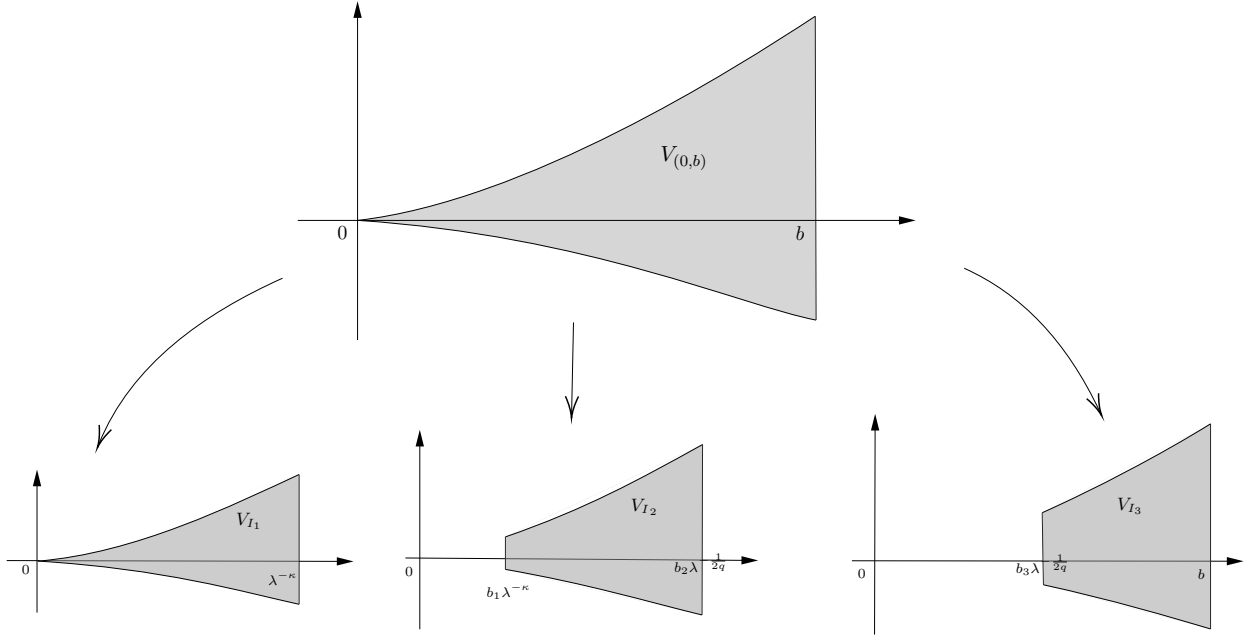
\begin{figure}[t]
    \centering
\scalebox{0.75}{
\begin{tikzpicture}[x=1pt,y=1pt,yscale=-1,xscale=1]

\draw    (429.6,57.6) -- (430,210) ;
\draw  [draw opacity=0][fill={rgb, 255:red, 155; green, 155; blue, 155 }  ,fill opacity=0.41 ] (429.6,57.6) -- (430,210) -- (204.67,159.47) -- cycle ;
\draw [fill={rgb, 255:red, 255; green, 255; blue, 255 }  ,fill opacity=1 ]   (200,160) .. controls (284,150.4) and (386.8,85.6) .. (429.6,57.6) ;
\draw [fill={rgb, 255:red, 255; green, 255; blue, 255 }  ,fill opacity=1 ]   (200.4,160) .. controls (315.2,167.6) and (392.8,202.8) .. (430,210) ;
\draw    (170,160) -- (478,160) ;
\draw [shift={(480,160)}, rotate = 180] [fill={rgb, 255:red, 0; green, 0; blue, 0 }  ][line width=0.08]  [draw opacity=0] (8.4,-2.1) -- (0,0) -- (8.4,2.1) -- cycle    ;
\draw    (200,220) -- (200,52) ;
\draw [shift={(200,50)}, rotate = 90] [fill={rgb, 255:red, 0; green, 0; blue, 0 }  ][line width=0.08]  [draw opacity=0] (8.4,-2.1) -- (0,0) -- (8.4,2.1) -- cycle    ;
\draw    (39.19,356.75) -- (39.19,286.96) ;
\draw [shift={(39.19,284.96)}, rotate = 90] [fill={rgb, 255:red, 0; green, 0; blue, 0 }  ][line width=0.08]  [draw opacity=0] (8.4,-2.1) -- (0,0) -- (8.4,2.1) -- cycle    ;
\draw    (24,330.64) -- (196,331.06) ;
\draw [shift={(198,331.07)}, rotate = 180.14] [fill={rgb, 255:red, 0; green, 0; blue, 0 }  ][line width=0.08]  [draw opacity=0] (8.4,-2.1) -- (0,0) -- (8.4,2.1) -- cycle    ;
\draw    (170.64,287.48) -- (170.64,352.57) ;
\draw  [draw opacity=0][fill={rgb, 255:red, 155; green, 155; blue, 155 }  ,fill opacity=0.5 ] (170.64,287.48) -- (170.64,352.57) -- (43.87,330.67) -- cycle ;
\draw [fill={rgb, 255:red, 255; green, 255; blue, 255 }  ,fill opacity=1 ]   (39.7,330.9) .. controls (88.61,326.99) and (137.83,302.62) .. (170.64,287.48) ;
\draw [fill={rgb, 255:red, 255; green, 255; blue, 255 }  ,fill opacity=1 ]   (39.7,330.9) .. controls (102.89,334.82) and (141.17,346.05) .. (170.64,352.57) ;
\draw    (231.22,365.08) -- (231.22,267) ;
\draw [shift={(231.22,265)}, rotate = 90] [fill={rgb, 255:red, 0; green, 0; blue, 0 }  ][line width=0.08]  [draw opacity=0] (8.4,-2.1) -- (0,0) -- (8.4,2.1) -- cycle    ;
\draw    (224.01,329.34) -- (399,330.57) ;
\draw [shift={(401,330.58)}, rotate = 180.4] [fill={rgb, 255:red, 0; green, 0; blue, 0 }  ][line width=0.08]  [draw opacity=0] (8.4,-2.1) -- (0,0) -- (8.4,2.1) -- cycle    ;
\draw    (274.18,318.76) -- (274.18,335.06) ;
\draw    (373.07,272.43) -- (372.78,357.93) ;
\draw  [draw opacity=0][fill={rgb, 255:red, 155; green, 155; blue, 155 }  ,fill opacity=0.5 ] (274.18,318.76) -- (329.3,292.94) -- (373.07,272.43) -- (372.78,357.93) -- (274.18,335.06) -- cycle ;
\draw [fill={rgb, 255:red, 255; green, 255; blue, 255 }  ,fill opacity=1 ]   (274.18,335.06) .. controls (316.27,342.49) and (343.66,350.79) .. (372.78,357.93) ;
\draw [fill={rgb, 255:red, 255; green, 255; blue, 255 }  ,fill opacity=1 ]   (274.18,318.76) .. controls (314.54,304.75) and (341.07,290.16) .. (373.07,272.43) ;
\draw    (459.36,367.75) -- (459.66,259) ;
\draw [shift={(459.67,257)}, rotate = 90.16] [fill={rgb, 255:red, 0; green, 0; blue, 0 }  ][line width=0.08]  [draw opacity=0] (8.4,-2.1) -- (0,0) -- (8.4,2.1) -- cycle    ;
\draw    (448,330.74) -- (642.5,330.97) ;
\draw [shift={(644.5,330.97)}, rotate = 180.07] [fill={rgb, 255:red, 0; green, 0; blue, 0 }  ][line width=0.08]  [draw opacity=0] (8.4,-2.1) -- (0,0) -- (8.4,2.1) -- cycle    ;
\draw    (543.49,306.63) -- (544.1,342.8) ;
\draw    (627.92,262.05) -- (627.92,364.95) ;
\draw  [draw opacity=0][fill={rgb, 255:red, 155; green, 155; blue, 155 }  ,fill opacity=0.5 ] (543.49,306.63) -- (627.92,262.05) -- (627.92,364.95) -- (544.1,342.8) -- cycle ;
\draw [fill={rgb, 255:red, 255; green, 255; blue, 255 }  ,fill opacity=1 ]   (543.49,306.63) .. controls (569.89,294.85) and (603.05,277.47) .. (627.92,262.05) ;
\draw [fill={rgb, 255:red, 255; green, 255; blue, 255 }  ,fill opacity=1 ]   (544.1,342.8) .. controls (578.8,350.37) and (607.04,358.78) .. (627.92,364.95) ;
\draw    (162.33,189) .. controls (122.93,206.73) and (86.44,232.22) .. (70.23,272.15) ;
\draw [shift={(69.5,273.98)}, rotate = 291.12] [color={rgb, 255:red, 0; green, 0; blue, 0 }  ][line width=0.75]    (10.93,-3.29) .. controls (6.95,-1.4) and (3.31,-0.3) .. (0,0) .. controls (3.31,0.3) and (6.95,1.4) .. (10.93,3.29)   ;
\draw    (504.33,184) .. controls (532.9,196.81) and (555.64,220.28) .. (572.56,256.34) ;
\draw [shift={(573.33,258)}, rotate = 245.32] [color={rgb, 255:red, 0; green, 0; blue, 0 }  ][line width=0.75]    (10.93,-3.29) .. controls (6.95,-1.4) and (3.31,-0.3) .. (0,0) .. controls (3.31,0.3) and (6.95,1.4) .. (10.93,3.29)   ;
\draw    (308.33,211.17) -- (307.38,257.17) ;
\draw [shift={(307.33,259.17)}, rotate = 271.19] [color={rgb, 255:red, 0; green, 0; blue, 0 }  ][line width=0.75]    (10.93,-3.29) .. controls (6.95,-1.4) and (3.31,-0.3) .. (0,0) .. controls (3.31,0.3) and (6.95,1.4) .. (10.93,3.29)   ;

\draw (189,162) node [anchor=north west][inner sep=0.75pt]  [font=\small] [align=left] {$\displaystyle 0$};
\draw (351,122) node [anchor=north west][inner sep=0.75pt]   [align=left] {$\displaystyle V_{( 0,b)}$};
\draw (419,162) node [anchor=north west][inner sep=0.75pt]  [font=\small] [align=left] {$\displaystyle b$};
\draw (223.01,332.84) node [anchor=north west][inner sep=0.75pt]  [font=\tiny] [align=left] {$\displaystyle 0$};
\draw (263.79,340.99) node [anchor=north west][inner sep=0.75pt]  [font=\tiny] [align=left] {$\displaystyle b_{1} \lambda ^{-\kappa }$};
\draw (357.87,325.98) node [anchor=north west][inner sep=0.75pt]  [font=\tiny] [align=left] {$\displaystyle b_{2} \lambda ^{-\frac{1}{2q}}$};
\draw (336.07,298.44) node [anchor=north west][inner sep=0.75pt]  [font=\footnotesize] [align=left] {$\displaystyle V_{I}{}_{_{2}}$};
\draw (450,333.74) node [anchor=north west][inner sep=0.75pt]  [font=\tiny] [align=left] {$\displaystyle 0$};
\draw (580.1,298.2) node [anchor=north west][inner sep=0.75pt]  [font=\footnotesize] [align=left] {$\displaystyle V_{I_{3}}$};
\draw (529.28,326.47) node [anchor=north west][inner sep=0.75pt]  [font=\tiny] [align=left] {$\displaystyle b_{3} \lambda ^{-\frac{1}{2q}}$};
\draw (620.65,334.08) node [anchor=north west][inner sep=0.75pt]  [font=\tiny] [align=left] {$\displaystyle b$};
\draw (154.13,334.87) node [anchor=north west][inner sep=0.75pt]  [font=\tiny] [align=left] {$\displaystyle \lambda ^{-\kappa }$};
\draw (30.63,332.04) node [anchor=north west][inner sep=0.75pt]  [font=\tiny] [align=left] {$\displaystyle 0$};
\draw (138.08,305.51) node [anchor=north west][inner sep=0.75pt]  [font=\footnotesize] [align=left] {$\displaystyle V_{I_{1}}$};

\end{tikzpicture}
}
    \caption{Partition of $V_{(0,b)}$}
    \label{fig:cut my life into pieces}
\end{figure}

\subsection{General upper bound for the magnetic Laplacian on a peak}
We will first establish an upper bound for $Q_\lambda^I$ for any bounded, open interval $I$. This is achieved by doing a change of variables that produces a unitary equivalent operator on the rectangle $\Pi_I$. We then project the operator on $\Pi_I$ onto $I$, which gives the desired upper bound. This is also how the one-dimensional model operators $T^I_\lambda$ and $T_\lambda$ that were introduced in Subsection \ref{sec: 1dim model op} come into play. 

Let us start with the change of variables.
\begin{lemma}
\label{lem: coordinate change}
    Let $I \subseteq (0,\infty)$ be an open interval. Then $Q^{I}_\lambda $ is unitary equivalent to the operator generated by 
    \[
         \ell^I_\lambda(v) :=  \int\limits_{\Pi_I} \left|s^{-q}\partial_tv \right|^2 + \left| \partial_s v - \frac{qt}{s} \partial_t v + i \lambda s^q t v  -\frac{q}{2s} v \right|^2 \ddd t \ddd s
    \]
    with its domain $D(\ell^I_\lambda)$ being the closure of $H^1_I(\Pi_I)$ in regards to the norm induced by $\ell^I_\lambda$.
\end{lemma}

\begin{proof}
    Consider the map
    \[
        \varphi : \R^2 \rightarrow \R^2, \quad (s,t) \mapsto (s,   s^q t)
    \]
    then one has $ \varphi(\Pi_I) = V_I $ and one can define the coordinate change
    \[
        \Phi_1 : L^2(V_I) \rightarrow L^2(\Pi_I , s^{q}), \quad f \mapsto f \circ \varphi =:u .
    \]
    Since one has 
    \[
        D  \varphi = 
        \begin{pmatrix}
        1 & 0 \\  q s^{q-1}t  &   s^q
        \end{pmatrix}, 
        \quad \det D\varphi = s^q ,
    \]
    it is easy to see that $\Phi_1$ is unitary:
    \[
    \| f\|^2_{L^2(V_I)} = \int\limits_{V_I} |f(x,y)|^2 \ddd  x \ddd y =
     \int\limits_{\Pi_I} | f(s,s^q t) |^2 s^q \ddd s \ddd t = \| \Phi_1 f \|^2_{L^2(\Pi_I, s^q)}.
    \]
    This allows us to do a change of variables $u(s,t) = f(s,ts^q)$ and we get
    \[
         \partial_s u = \partial_x f + qs^{q-1}t \partial_y f, \quad 
        \partial_t u = s^q \partial_y f.
    \]
    Thus,
    \[
        \partial_y f = s^{-q} \partial_t u, \quad 
        \partial_x  f = \partial_s u - qts^{-1} \partial_t u 
    \]
    and we write
    \begin{align*}
        q^{I}_\lambda (f) &= \int\limits_{V_I} | \partial_x f|^2 + | \partial_y f- i \lambda x f|^2 \ddd  x \ddd y  \\
        &= \int\limits_{\Pi_I} \left( \left| s^{-q} \partial_t u - i\lambda s u \right|^2 + \left|\partial_s u - qs^{-1}t \partial_t u\right|^2 \right) s^q \ddd s \ddd t =: \tilde{q}_\lambda(u).
    \end{align*}
    The next step is to eliminate the weight $s^q$ with the help of another map 
    \[
    \Phi_2 :L^2(\Pi_I) \rightarrow L^2(\Pi_I,s^q). \quad \Phi_2(u) = s^{-\frac{q}{2}}u .
    \]
    $\Phi_2$ is unitary:
    \[
    \| \Phi_2 u \|^2_{L^2(\Pi_I,s^q)} 
    =\int\limits_{\Pi_I} \left|s^{-\frac{q}{2}} u\right|^2 s^q\ddd s \ddd t 
    = \int\limits_{\Pi_I} |u|^2 \ddd s\ddd t = \| u \|^2_{L^2(\Pi_I)}.
    \]
    And we have
    \[
     \partial_s (s^{-\frac{q}{2}} u) = s^{-\frac{q}{2}} \partial_s u - \frac{q}{2 } s^{-\frac{q}{2} -1} u = s^{-\frac{q}{2}} (\partial_s u- \frac{q}{2s}u), \quad  
     \partial_t (s^{-\frac{q}{2}} u) = s^{-\frac{q}{2}} (\partial_t u) .
    \]
    Therefore,
    \begin{align*}
        \tilde{q}_\lambda (\Phi_2 u) &= \int\limits_{\Pi_I} s^q \left( \left| s^{-q} \partial_t (s^{-\frac{q}{2}}u) - i\lambda s (s^{-\frac{q}{2}}u) \right|^2 + \left|\partial_s (s^{-\frac{q}{2}}u) - qs^{-1}t \partial_t (s^{-\frac{q}{2}}u)\right|^2 \right) \ddd s \ddd t \\
         &=\int\limits_{\Pi_I} \left| s^{-q} \partial_t u - i\lambda s u \right|^2 + \left|\partial_s u - \frac{q}{2s}u - \frac{qt}{s}\partial_t u \right|^2  \ddd s \ddd t =: q'_\lambda (u).
    \end{align*}
    Lastly, we apply a gauge transformation $\Phi_3$, with $\Phi_3(v)(s,t) = e^{i\lambda s^{q+1}t} v(s,t) = u(s,t)$, and get
    \[
    \partial_s u = e^{i \lambda s^{q+1}t}(i \lambda (q+1)s^q t \,v + \partial_s v), \quad 
    \partial_t u = e^{i \lambda s^{q+1}t} (i \lambda s^{q+1} \,v + \partial_t v).
    \]
    Thus,
    \begin{align*}
        q'_\lambda (u) &= \int\limits_{\Pi_I} \left|s^{-q}\partial_tv + i \lambda sv - i \lambda s v \right|^2 + \left| \partial_s v - \frac{qt}{s} \partial_t v + i \lambda (q+1)s^q t v - i \lambda qs^qtv -\frac{q}{2s} v \right|^2 \ddd t \ddd s \\
        &=  \int\limits_{\Pi_I} \left|s^{-q}\partial_tv \right|^2 + \left| \partial_s v - \frac{qt}{s} \partial_t v + i \lambda s^q t v  -\frac{q}{2s} v \right|^2 \ddd t \ddd s =: \ell^I_\lambda(v).
    \end{align*}
    Due to $|e^{i\lambda s^{q+1} t}| =1$, $\Phi_3$ is unitary and therefore $ \Phi:= \Phi_3^{-1} \circ \Phi_2^{-1} \circ \Phi_1:L^2(V_I) \to L^2(\Pi_I) $ is unitary as well. Note that
    \[
    \Phi(H^1_I(V_I)) 
    = H^1_{I}(\Pi_I).
    \]
    Thus, the operator $L_\lambda^I$ generated by $\ell_\lambda^I$ with domain $D(\ell_\lambda^I)$, the closure of $H^1_I(\Pi_I)$ with respect to the norm induced by $\ell_\lambda^I$, is unitary equivalent to $Q_\lambda^I$.
 \end{proof}

We will now prove an upper bound for the eigenvalues of $Q_\lambda^I$ 
with bounded $I$, which will be related to the one-dimensional model operator $T_\lambda$ introduced in subsection \ref{sec: 1dim model op} by considering $L_\lambda^I$ restricted to the space spanned by functions that only depend on $s \in I$.
\begin{lemma}
\label{lem: upper bound corner}
    Let $I \subseteq  (0,\infty)$ be an open, bounded interval. Consider the quadratic form in $L^2(I)$
    \[
    \tilde{t}^I_\lambda (g) = \int\limits_I \left|g'+ i \lambda s^q \frac{1}{2} (a_++a_-) g \right|^2 + |g|^2\left( \frac{(a_+-a_-)^2}{12} \lambda^2s^{2q}+ \frac{q(q-2)}{4s^2}\right) \ddd s, \quad D(\tilde t_\lambda^I) = H^1 _0(I).
    \]
    Then $T_\lambda^I$ is unitary equivalent to $\tilde T_\lambda^I$ and it holds for any $n \in \N$
    $$ E_n(L^I_\lambda) \leq E_n (\tilde{T}^I_\lambda) = E_n(T_\lambda^I). $$
\end{lemma}
\begin{proof}

The unitary equivalence of $T_\lambda^I$ and $\tilde T_\lambda^I$ follows directly from the gauge transformation 
\[
\Phi: L^2(I) \to L^2(I), \quad h(s) \mapsto e^{-i\lambda s^{q+1} \frac{a_++a_-}{2(q+1)}} h(s) =: g(s).
\]
It is straightforward to show that this operator is unitary, and it holds $ \tilde t_\lambda^I (g) = t_\lambda^I(h)$ as well as $D(\tilde t_\lambda^I) = D(t_\lambda^I) = H^1_0(I)$.

For the comparison of $L_\lambda^I$ and $\tilde T_\lambda^I$, we want to apply Lemma \ref{comparing operators} so we introduce the operator
\[
J: C^\infty_c(I) \to D(\ell_\lambda^I), \quad g \mapsto (a_+-a_-)^{-\frac{1}{2}}(g  \otimes 1),
\]
that is well-defined due to $J(C_c^\infty(I)) \subset H^1_I(\Pi_I) \subset D(L_\lambda^I)$. One also has for any $ g \in C^\infty_c(I)$:
\[
\|Jg \|^2_{L^2(\Pi_I)} = \int\limits_I \int\limits_{a_-}^{a_+} \frac{1}{a_+-a_-} |g|^2\ddd t\ddd s =\frac{1}{a_+-a_-}\int\limits_I |g|^2 \ddd s \int\limits_{a_-}^{a_+}1\ddd t= \| g\|^2_{L^2(I)}.
\]
The last condition needed for Lemma \ref{comparing operators} holds as well, as one can write for any $ g \in C^\infty_c(I)$:
    \begin{align*}
    \ell^I_\lambda(Jg) =& \frac{1}{a_+-a_-} \int\limits_{\Pi_I} \left| g' + g \left(i \lambda s^q t - \frac{q}{2s}\right) \right|^2  \ddd t \ddd s \\
    =& \frac{1}{a_+-a_-} \int\limits_{I} \int\limits_{a_-}^{a_+} |g'|^2 + |g|^2\left(\lambda^2s^{2q}t^2 + \frac{q^2}{4s^2}\right) + g'\bar g\left(- i \lambda s^q t - \frac{q}{2s}\right) + \bar g'g \left(i \lambda s^q t - \frac{q}{2s}\right) \ddd t \ddd s \\
    =& \int\limits_I |g'|^2 + |g|^2 \left( \frac{a_+^2 + a_-a_+ +a_-^ 2}{3}\lambda^2 s^{2q} + \frac{q^2}{4s^2}\right) \\
    &- \frac{q}{2s} (|g|^2)' + \bar g' \left(gi\lambda s^q \frac{a_++a_-}{2}\right) + g'\overline{\left(gi\lambda s^q \frac{a_++a_-}{2}\right)} \ddd s \\
    =&\int\limits_I \left|g'+ i \lambda s^q \frac{1}{2} (a_++a_-) g \right|^2 + |g|^2\left( \frac{(a_+-a_-)^2}{12} \lambda^2s^{2q}+ \frac{q(q-2)}{4s^2}\right) \ddd s
    = \tilde t^I_\lambda(g). \qedhere
    \end{align*}
    
\end{proof}

By combining this Lemma with Lemma \ref{lem: coordinate change} and Lemma \ref{lem: model operator comparison cropped}, we get a concrete upper bound for $Q_\lambda^I$ for a specific choice of $I$.

\begin{corollary}
\label{cor: upper bound corner}
Let $I = (0,\lambda^{-\kappa}) $ and $n \in \N$ then there exists $C>0$ such that
    \[
    E_n(Q_\lambda^I) \leq \lambda^{\frac{2}{q+1}} \left(E_n(T_1)+ C \lambda^{2(\kappa - \frac{1}{q+1})}\right)
    \]
    holds as $\lambda \to + \infty$.
\end{corollary}

\subsection{Lower bound for the magnetic Laplacian at the peak}
In this subsection we assume that $I = (0, \lambda^{-\kappa{}})$ holds for some $\kappa \in (0, \frac{1}{q+1}) $. With the help of the one-dimensional model operator $T_\lambda$, we will establish a suitable lower bound for the eigenvalues of $Q_\lambda^I$. We start by comparing $L_\lambda^I$ and $T_\lambda^I$.
\begin{lemma}
\label{lem: lower bound corner}
    Let $n \in \N$. For any $\kappa \in (\frac{2q+1}{(q+1)3q}, \frac{1}{q+1})$ there exists $C>0$ such that
    \[
        E_n(L^I_\lambda) \geq  E_n(T^I_\lambda) (1 - C \lambda^{ \frac{2q+1}{q+1}-3q\kappa} ) 
    \]
    holds as $\lambda  \to +\infty$.
\end{lemma}
\begin{proof}
    Consider the orthogonal projection 
    \begin{equation}
    \label{eq: Pi 0}
    P_0: L^2(\Pi_I) \to  L^2(\Pi_I) ,\quad 
      (P_0v) (s,t) = \frac{1}{a_+-a_-} \int\limits_{a_-}^{a_+} v (s,t) \ddd t.
    \end{equation}
    Note that functions in the image of $P_0$ are independent of $t$. We will therefore often treat a function $P_0 v \in L^2(\Pi_I)$ as 
    as a function in $L^2(I)$ which will be indicated by the given norm.

    Let $\psi_1, \dots, \psi_n$ be the first $n$ orthonormal eigenfunctions of $L^I_\lambda$ and set $V := \spann \{ \psi_1, \dots,\psi_n \}$ as well as $ V_0:= P_0 V \subseteq L^2(I)$. 
    For any function $u \in L^2(\Pi_I)$ we set  $\tilde u := u - P_0u$. Note that due to $P_0$ being an orthogonal projection, one has $ \|u\|^2 = \|P_0u\|^2 + \| \tilde u\|^2 $ in the $L^2(\Pi_I)$ norm.

    First, we want to show that $V_0$ is an $n-$dimensional subspace of $H^1_0(I)$. To achieve this, we start by proving some basic properties and norm estimates that hold for any $u \in L^2(\Pi_I)$ and $v \in  V$.

    For any $g \in P_0(L^2(\Pi_I))$ one can rewrite $\langle g, \tilde u \rangle_{L^2(\Pi_I)}$ the following way:
    \begin{equation*}
    \int\limits_I \int\limits_{a_-}^{a_+} \bar g (s) \left(u(s,t) - \frac{1}{a_+-a_-} \int\limits_{a_-}^{a_+} u(s,\tau) \ddd  \tau \right) = \int\limits_I g(s) \left(\int\limits_{a_-}^{a_+} u(s,t) \ddd t -  \int\limits_{a_-}^{a_+} u(s,\tau ) \ddd \tau\right) \ddd s = 0.
    \end{equation*}
    Analogously, it holds $ \langle g, \partial_s \tilde u \rangle_{L^2(\Pi_I)} = 0$. 
    Let $v \in V$. Due to $ \partial_t \tilde v = \partial_t v$, the spectral theorem for Neumann Laplacian on $(a_-,a_+)$ implies
    \begin{equation}
    \label{eq: lower bound partial t tilde v}
         \|\partial_t \tilde v\|^2_{L^2(\Pi_I)} \geq E_2 (-\Delta^N_{a_-,a_+})  \| \tilde v\|^2_{L^2(\Pi_I)},
     \end{equation}
    where $- \Delta^N_{a_-,a_+}$ is the Neumann Laplacian on the interval $(a_-,a_+)$.
    
    On the other hand, we have
    
    \begin{align*}
    \begin{split}
     \|\partial_t v\|^2_{L^2(\Pi_I)} 
    &= \int\limits_I \int\limits_{a_-}^{a_+} \lambda^{2q\kappa} \lambda^{-2q\kappa} |\partial_t v|^2 \ddd t \ddd s 
    \leq \lambda^{-2q\kappa} \int\limits_I \int\limits_{a_-}^{a_+} |s^{-q } \partial_t v |^2 \ddd t \ddd s \\
    &\leq \lambda^{-2q\kappa} \ell^I_\lambda(v) 
    \leq  \lambda^{-2q\kappa} E_n(L^I_\lambda) \|v\|^2_{L^2(\Pi_I)},
    \end{split}
    \end{align*}
    
    which can be combined with (\ref{eq: lower bound partial t tilde v}) and $\partial_t \tilde v = \partial_tv $ to get for $c_0  =(E_2(-\Delta_{a_-,a_+}))^{-1/2} >0$:
    \begin{equation}
    \label{eq: upper bound tilde v}
        \| \tilde v\|_{L^2(\Pi_I)} 
        \leq \frac{\lambda^{-q\kappa} \sqrt{E_n(L^I_\lambda) } \|v \|_{L^2(\Pi_I)} }{\sqrt{E_2(-\Delta_{a_-,a_+}^N)}}
        = c_0\lambda^{-q\kappa} \sqrt{E_n(L^I_\lambda)} \|v \|_{L^2(\Pi_I)}.
    \end{equation}

    This allows us to find both a lower and an upper bound for $P_0v$:

    \begin{equation}
    \label{eq: lower bound v0}
         \| v\|^2_{L^2(\Pi_I)} \bigl( 1 - c_0^2 \lambda^{-2q\kappa} {E_n(L^I_\lambda)} \bigr) \leq \| P_0 v\|_{L^2(\Pi_I)}^2 \leq \| v\|^2_{L^2(\Pi_I)}.
    \end{equation}
    We use Corollary \ref{cor: upper bound corner} and by ensuring that 
    \[
    \lambda^{-2q\kappa}E_n(L_\lambda^I ) \leq \bigl(E_n(T_1)+1 \bigr)\lambda^{-2q\kappa +\frac{2}{q+1}} = \bigl(E_n(T_1)+1 \bigr) \lambda^{2(\frac{1}{q+1}-q \kappa)} \to 0
    \]
     holds as $\lambda \to +\infty$, it follows that $V_0$ is an $n-$dimensional subspace of $L^2(I)$. This is given as per our assumption we have
     \[
    \kappa > \frac{2q+1}{3q(q+1)} > \frac{1}{q(q+1)}  .
     \]
     
    We still need to show that $ P_0(D(\ell^I_\lambda)) \subseteq H^1_0(I)$ holds. 
    Consider the space of functions 
    \[
    D :=  \{ w \in C^\infty(\Pi_I) \mid \exists \, c_0, c_1 \in I, \, c_0 < c_1 : w(s,t) = 0 , \,\forall \, (s,t) \in \Pi_I  \text{ with } s \notin [c_0,c_1] \}
    \]
    that is dense in $D(\ell_\lambda^I)$ due to being dense in $H^1_I(\Pi_I)$.
    It is apparent that $P_0 w \in C^\infty_c(I)$ holds for any $w \in D$. As $C^\infty_c(I)$ is dense in $H^1_0(I)$ as well, it remains to prove that for any $w \in D$ there exists a constant $C_\lambda>0$ only dependent on $\lambda$ such that
    \[
    \| P_0 w \|^2_{H^1(I)} \leq C_\lambda \| w \|^2_{\ell}:= C_\lambda \left(\|w \|^2_{L^2(\Pi_I)} + \ell_\lambda^I(w) \right).
    \]
    Let $w \in D$. Note that for any fixed $ t\in (a_-,a_+) $ one has $w( \cdot,t) \in C^\infty_c(I)$ and one may apply the one-dimensional Hardy inequality (\ref{eq: Hardy}). 
    By combining this with the inequalities $|x+y|^2 \geq (1-\varepsilon)|x|^2 - |y|^2/\varepsilon$ and $2|xy| \leq \varepsilon|x|^2+|y|^2/ \varepsilon  $ that hold for any $\varepsilon>0$ we get 
    
    \begin{align*}
        \ell_\lambda^I(w) =&  \int\limits_{a_-}^{a_+} \int\limits_I |s^{-q} \partial_t w |^2+ |\partial_s w |^2 + \left|\frac{qt}{s} \partial_t w + i\lambda s^q t w -\frac{q}{2s}w \right|^2 \\ 
        &+ \overline{\partial_s w} \left(\frac{qt}{s} \partial_t w + i\lambda s^q t w -\frac{q}{2s}w \right) + \partial_s w \overline{\left(\frac{qt}{s} \partial_t w + i\lambda s^q t w -\frac{q}{2s}w \right)} \ddd s \ddd t \\
        \geq& \int\limits_{a_-}^{a_+} \int\limits_I |s^{-q} \partial_t w |^2+ |\partial_s w |^2 + |w|^2 \left( (1- \varepsilon)\frac{q^2 }{4s^2}-\frac{q}{2s^2} \right) - \frac{1}{\varepsilon} \left|\frac{qt}{s} \partial_t w + i \lambda s^q t w \right|^2 \\
        &-2 |\partial_sw| \left|\frac{qt}{s} \partial_t w + i \lambda s^q t w \right|\ddd s \ddd t \\
        \geq& \int\limits_I \int\limits_{a_-}^{a_+} |s^{-q} \partial_t w |^2+ |\partial_s w |^2(1-\varepsilon) + |w|^2 \left( (1- \varepsilon)\frac{q^2 }{4s^2}-\frac{q}{2s^2} \right) - \frac{2}{\varepsilon}\left|\frac{qt}{s} \partial_t w + i \lambda s^q t w \right|^2 \ddd t\ddd s \\
        \geq& \int\limits_I \int\limits_{a_-}^{a_+} | s^{-q} \partial_t w |^2 \left( 1- s^{2(q-1)}\frac{4q^2t^2}{\varepsilon} \right)+ \frac{\varepsilon}{2}|\partial_s w |^2 \\
        &+ |w|^2 \frac{(q-1)^2-\varepsilon q^2-2\varepsilon}{4s^2}- |w|^2 4\lambda^2 s^{2q} t^2/\varepsilon \ddd t\ddd s .
    \end{align*}
    By choosing $\varepsilon$ sufficiently small there exists $\lambda_0>0$ such that for all $\lambda \geq \lambda_0$ it holds
    \[
    \ell_\lambda^I(w) \geq \frac{\varepsilon}{2}\| \partial_s w\|^2_{L^2(I)} - \bigl(2\lambda^{2-2q\kappa}a_{max}^2/\varepsilon\bigr)\| w\|^2_{L^2(I)}.
    \]
    Therefore, it is clear that there is a $c_\lambda >0$ such that
    \[
     \|\partial_s P_0w  \|^2_{L^2(I)} \leq  (a_+-a_-)^{-1} \| \partial_s w\|^2_{L^2(\Pi_I)} \leq c_\lambda\ell^I_\lambda(w) + c_\lambda \|w \|^2_{L^2(\Pi_I)} = c_\lambda \| w \|^2_{\ell}
    \]
    holds. Thus, we have $ \|P_0w\|^2_{H^1(I)} \leq ( (a_+-a_-)^{-1}+c_\lambda)\|w\|^2_\ell $ which concludes the proof that $V_0$ is an $n$-dimensional subspace of $H^1(I)$.
    
    The Min-Max principle, the inequality (\ref{eq: lower bound v0}) and the calculations done in Lemma \ref{lem: upper bound corner} then yield
    \begin{align}
    \label{eq: Min Max T I lambda}
    \begin{split}
        E_n(T^I_\lambda) 
        &\leq \sup\limits_{ \underset{g \neq 0}{g \in V_0}} \frac{(a_+-a_-)\tilde t^I_\lambda(g)}{(a_+-a_-)\|g\|^2_{L^2(I)}} 
        = \sup\limits_{ \underset{v \neq 0}{v \in V} } \frac{\ell^I_\lambda(P_0 v)}{\|P_0 v\|^2_{L^2(\Pi_I)}} \\
        &\leq \sup\limits_{ \underset{v\neq 0}{v \in V} } \frac{\ell^I_\lambda(P_0 v)}{\| v\|^2_{L^2(\Pi_I)}(1-c^2_0\lambda^{-2q\kappa}E_n(L^I_\lambda))} .
    \end{split}
    \end{align}
    So our next goal is the approximation of $C_\lambda:= \ell^I_\lambda(P_0 v)$. It holds
    \begin{equation}
    \label{eq: Clambda leq}
        \ell^I_\lambda(v) = \ell^I_\lambda(P_0 v) + \underbrace{\ell^I_\lambda (\tilde v)}_{\geq 0} + 2 \Re(\ell^I_\lambda(P_0 v, \tilde v)) \quad \Rightarrow \quad 
         C_\lambda \leq E_n(L^I_\lambda)\|v\|^2_{L^2(\Pi_I)}  + 2 | \ell_\lambda (P_0 v,\tilde v) |.
    \end{equation}
    Thus, the next step is to estimate $\ell^I_\lambda(P_0 v, \tilde v)$. Due to the orthogonality of $\tilde v$ and $\partial_s \tilde v$ to functions in $P_0(V) \subseteq L^2(\Pi_I)$ and  partial integration, one has
\begin{align*}
     \ell&^I_\lambda (P_0 v, \tilde v) = \int\limits_I \int\limits_{a_-}^{a_+} \left[( \partial_s \overline{P_0 v} - \frac{q}{2s} \overline{P_0 v}) - i\lambda s^q t\overline{P_0 v} \right] \left[(\partial_s \tilde v  - \frac{q}{2s} \tilde v) - \frac{qt}{s} \partial_t \tilde v+ i \lambda s^q t \tilde v\right] \ddd t \ddd s \\
     =& \int\limits_{I} \int\limits_{a_-}^{a_+} \partial_s \tilde v \left(\partial_s \overline{P_0 v} - \frac{q}{2s} \overline{P_0v}\right) + \tilde v \left(\frac{q^2}{4s^2}\overline{P_0 v}-\frac{q}{2s}\partial_s \overline{P_0 v} \right) - i \lambda s^q t \overline{P_0 v} \partial_s \tilde v   +i\lambda s^{q-1} t \frac{q}{2} \overline{P_0 v} \tilde v \\
     &-\frac{qt}{s} \partial_t \tilde v \left(\partial_s \overline{P_0 v} - \frac{q}{2s} \overline{P_0 v} - i\lambda s^q t\overline{P_0 v}\right) + i \lambda s^q t \partial_s \overline{P_0 v} \tilde v -i \lambda s^{q-1}t\frac{q}{2}\overline{P_0 v} \tilde v + \lambda^2 s^{2q}t^2 \overline{P_0 v} \tilde v \, \ddd t \ddd s \\
     =& \int\limits_I \int\limits_{a_-}^{a_+}  -\frac{qt}{s} \partial_t \tilde v \left(\partial_s \overline{P_0 v} - \frac{q}{2s} \overline{P_0 v} - i\lambda s^q t\overline{P_0 v}\right) + i \lambda s^q t \partial_s \overline{P_0 v} \tilde v  - i \lambda s^q t \overline{P_0 v} \partial_s \tilde v  + \lambda^2 s^{2q}t^2 \overline{P_0 v} \tilde v \, \ddd t \ddd s \\
     =&  \int\limits_I \int\limits_{a_-}^{a_+}  -\frac{qt}{s} \partial_t \tilde v \left(\partial_s \overline{P_0 v} - \frac{q}{2s} \overline{P_0 v} - i\lambda s^q t\overline{P_0 v}\right) + 2i \lambda s^q t \partial_s \overline{P_0 v} \tilde v  + i \lambda qs^{q-1} t \overline{P_0 v} \tilde v  + \lambda^2 s^{2q}t^2 \overline{P_0 v} \tilde v \, \ddd t \ddd s \\
     =&\int\limits_I \int\limits_{a_-}^{a_+}  \left( i \lambda s^q t \tilde v -\frac{qt}{s} \partial_t \tilde v\right) \left(\partial_s \overline{P_0 v} - \frac{q}{2s} \overline{P_0 v} - i\lambda s^q t\overline{P_0 v}\right) + i \lambda s^q t \partial_s \overline{P_0 v} \tilde v  + i \frac{3}{2} \lambda qs^{q-1} t \overline{P_0 v} \tilde v   \, \ddd t \ddd s.
\end{align*}
    Using Cauchy-Schwarz, one gets
    \begin{align}
    \label{eq: Cauchy Schwarz}
    \begin{split}
    |\ell_\lambda^I(P_0 v, \tilde v)| 
    &\leq \left(a_{max}\lambda^{1-q\kappa} \|\tilde v\|_{L^2(\Pi_I)} +\left\|\frac{qt}{s} \partial_t \tilde v \right\|_{L^2(\Pi_I)}\right)  \left\|\partial_s P_0 v+ P_0 v(i\lambda s^q t-\frac{q}{2s} ) \right\|_{L^2(\Pi_I)} \\
    &+ a_{max}\lambda^{1-q\kappa} \| \partial_s P_0 v\|_{L^2(\Pi_I)} \|\tilde v\|_{L^2(\Pi_I) } + \frac{3}{2}a_{max} q \lambda^{1-(q-1)\kappa} \|P_0 v\|_{L^2(\Pi_I)} \|\tilde v\|_{L^2(\Pi_I)},
    \end{split}
    \end{align}
    and we can estimate the terms that are left the following way:
    \begin{align}
    \label{eq: qt:s partial t v upper bound}
    \begin{split}
        \left\| \frac{qt}{s} \partial_t \tilde v \right\|^2_{L^2(\Pi_I)} 
        &= \int\limits_I \int\limits_{a_-}^{a_+} \frac{q^2t^2}{s^2} |\partial_t v|^2 \ddd t \ddd s 
        \leq q^2a_{max}^2 \lambda^{-2(q-1)\kappa} \int\limits_I \int\limits_{a_-}^{a_+} s^{-2q} | \partial_t v|^2 \ddd t \ddd s \\
        &\leq q^2a_{max}^2 \lambda^{-2(q-1)\kappa}\ell^I_\lambda(v) 
        \leq q^2a_{max}^2 \lambda^{-2(q-1)\kappa}E_n(L^I_\lambda) \| v \|^2_{L^2(\Pi_I)},
    \end{split}
    \end{align}
    as well as
    \begin{equation}
    \label{eq: Clambda =}
         \left\|\partial_s P_0 v+ (i\lambda s^q t-\frac{q}{2s}) P_0 v \right\|^2_{L^2(\Pi_I)} = \ell^I_\lambda(P_0 v) = C_\lambda.
    \end{equation}
    By employing the inequality $|x+y|^2 \geq (1-\delta)x^2 - y^2/\delta $ for some $\delta \in (0,\min\{1/2,(q-1)^2/2\})$, and using the one-dimensional Hardy inequality (\ref{eq: Hardy}) one has
\begin{align*}
    (a_+-a_-)^{-1}\ell_\lambda^I (P_0 v) &= \tilde t^I_\lambda (P_0v) \geq \int\limits_I |\partial_s P_0v + i \lambda s^q \frac{1}{2} (a_++a_-) P_0 v|^2 +|P_0 v|^2 \frac{q(q-2)}{4s^2} \ddd s \\
     &\geq \int\limits_I \delta |\partial_sP_0 v|^2 - \lambda^2s^{2q} \frac{(a_++a_-)^2}{4\delta} |P_0 v|^2 + (1-2\delta) |\partial_sP_0 v|^2 + |P_0 v|^2 \frac{q(q-2)}{4s^2} \ddd s \\
     &\geq \int\limits_I \delta |\partial_sP_0 v|^2 + |P_0 v|^2 \frac{(q-1)^2 - 2\delta}{4 s^2} \ddd s - \frac{a_{max}^2}{\delta}\lambda^{2(1-q\kappa)} \|P_0  v\|^2_{L^2(I)} \\
     &\geq \delta \|\partial_s P_0v\|^2_{L^2(I)} - \frac{a_{max}^2}{\delta}\lambda^{2(1-q\kappa)}  \|P_0 v\|^2_{L^2(I)}.
\end{align*}
    Thus, there exists $c_1 >0$ such that
    \begin{equation}
    \label{eq: I hate you 1}
      \| \partial_s P_0 v\|_{L^2(\Pi_I)} \leq c_1 \lambda^{1-q\kappa} \|P_0 v\|_{L^2(\Pi_I)} +c_1\sqrt{C_\lambda}.     
    \end{equation}
    One also has, due to Lemma \ref{lem: model operator comparison cropped}:
    \[
    \ell_\lambda^I(P_0 v) = (a_+-a_-)\tilde t_\lambda^I(P_0 v) \geq (a_+-a_-)E_1( T_\lambda^I) \|P_0v\|^2_{L^2(I)} \geq 
    \lambda^{\frac{2}{q+1}}E_1(T_1) \|P_0v \|^2_{L^2}.
    \]
    Therefore
    \begin{equation}
    \label{eq: Pi0 v leq }    
        \| P_0 v \|_{L^2(\Pi_I)} \leq \lambda^{-\frac{1}{q+1}} (E_1(T_1))^{-\frac{1}{2}}\sqrt{C_\lambda}.
    \end{equation}
    By substituting (\ref{eq: upper bound tilde v}), (\ref{eq: qt:s partial t v upper bound}) ,(\ref{eq: Clambda =}), (\ref{eq: I hate you 1}), and (\ref{eq: Pi0 v leq }) into (\ref{eq: Cauchy Schwarz}), we get for some $c_2 >0$ and sufficiently large $\lambda$:
    \begin{align*}
    | \ell_\lambda^I(P_0v,\tilde v) | &\leq 
    \sqrt{C_\lambda} \left((a_{max} c_0 + a_{max} c_0c_1)\lambda^{1-2q\kappa}   
    + q^2a_{max}^2\lambda^{-2(q-1)\kappa} \right) \sqrt{E_n(L_\lambda^I)} \| v\|_{L^2(\Pi_I)}  \\
    &\left(+ a_{max} c_0 c_1 \lambda^{2-3q\kappa} 
    + \frac{3}{2}a_{max} q c_0 \lambda^{1-(2q-1)\kappa} \right) \sqrt{E_n(L_\lambda^I)} \| v \|_{L^2(\Pi_I)}  \| P_0 v\|_{L^2(\Pi_I)} \\
    &\leq \sqrt{C_\lambda} \left( c_2 \lambda^{1-2q\kappa} + c_2\lambda^{-2(q-1)\kappa} + c_2 \lambda^{\frac{2q+1}{q+1}-3q\kappa} + c_2 \lambda^{\frac{q}{q+1}-(2q-1)\kappa} \right) \sqrt{E_n(L_\lambda^I)} \|v \|^2_{L^2(\Pi_I)}
    \end{align*}
    Since we have 
    \begin{itemize}
        \item $1-2q\kappa > -(2q-1)\kappa$ $\Leftrightarrow 1>2\kappa$ due to $\kappa < \frac{1}{q+1}<\frac{1}{2}$
        \item $ \frac{2q+1}{q+1} - 3q\kappa >1-2q\kappa $ which is equivalent to $ \kappa < \frac{1}{q+1}$
        \item $\frac{2q+1}{q+1}-3q\kappa > \frac{q}{q+1}-(2q-1)\kappa$ which is also equivalent to $ \kappa < \frac{1}{q+1} $
    \end{itemize}
    we get in combination with (\ref{eq: Clambda leq}) for sufficiently large $\lambda$:
    \[
    C_\lambda \leq E_n(L_\lambda^I) +  2c_2\lambda^{\frac{2q+1}{q+1}-3q\kappa} \sqrt{E_n(L_\lambda^I)} \|v \|_{L^2(\Pi_I)} \sqrt{C_\lambda} .
    \]
    Using the fact that for any $a,b,x \geq 0$, it holds
    \begin{equation}
    \label{eq: abx}
    x^2 \leq a + bx \quad \Rightarrow \quad  x \leq  \sqrt{a+ \frac{b^2}{4}} + \frac{b}{2} \quad \Rightarrow \quad 
    x^2 \leq a+b^2+\sqrt{a}b,
    \end{equation}
    we can write (with $x^2 = C_\lambda$)
    \begin{align*}
    C_\lambda &\leq E_n(L^I_\lambda) \| v\|^2_{L^2(\Pi_I)} \left(1+2c_2\lambda^{\frac{2q+1}{q+1}-3q\kappa} + 4c_2^2 \lambda^{2(\frac{2q+1}{q+1}-3q\kappa)} \right) .
    \end{align*}
    Due to  our choice of $\kappa > \frac{2q+1}{3q(q+1)}$ the terms containing powers of $\lambda$ tend to $0$ as $\lambda \to +\infty$ and, using Corollary \ref{cor: upper bound corner} it holds
    \[
    C_\lambda \leq E_n(L_\lambda^I) \|v\|^2_{L^2(\Pi_I)} + E_n(T_\lambda^I) \|v\|^2_{L^2(\Pi_I)}4c_2\lambda^{ \frac{2q+1}{q+1}-3q\kappa}.
    \]
     By combining this with (\ref{eq: Min Max T I lambda}), we arrive at
    \[
    E_n(T_\lambda^I) \leq \sup\limits_{\underset{v\neq0}{v \in V}} \frac{ E_n(L_\lambda^I) \|v\|^2_{L^2(\Pi_I)} + E_n(T_\lambda^I) \|v\|^2_{L^2(\Pi_I)}4c_2\lambda^{ \frac{2q+1}{q+1}-3q\kappa}}{\left(1-c_0^2\lambda^{2(\frac{1}{q+1}-q\kappa)}\right)\|v\|^2_{L^2(\Pi_I)}}
    = \frac{E_n(L_\lambda^I) + E_n(T_\lambda^I)4c_2\lambda^{ \frac{2q+1}{q+1}-3q\kappa}}{\left(1-c_0^2 \lambda^{2(\frac{1}{q+1}-q\kappa)}\right)}.
    \]
    Noting that one has
    \[
    2(\frac{1}{q+1}-q\kappa )< \frac{2q+1}{q+1} - 3q\kappa \Leftrightarrow q\kappa < \frac{2q-1}{q+1} \Leftrightarrow \kappa < \frac{1}{q+1} \frac{2q-1}{q},
    \]
    we conclude with
    \[
    E_n(T_\lambda^I)\left(1-c_0^2\lambda^{2(\frac{1}{q+1}-q\kappa)} - 4c_2\lambda^{ \frac{2q+1}{q+1}-3q\kappa}\right) \leq E_n(L_\lambda^I). \qedhere
    \]

\end{proof}

By combining this Lemma, Lemma \ref{lem: model operator comparison cropped}, and Lemma \ref{lem: coordinate change} we get a lower bound for the eigenvalues of $Q_\lambda^I$:

\begin{corollary}
\label{cor: lower bound corner}
Let $n \in \N$ and $I= (0,\lambda^{-\kappa})$. Then for any $\kappa \in (\frac{2q+1}{(q+1)3q}, \frac{1}{q+1}) $ there exists $C>0$ such that
\[
E_n(Q_\lambda^I) \geq \lambda^{\frac{2}{q+1}}E_n(T_1)-C\lambda^{\frac{2q+3}{q+1} -3q\kappa}
\]
holds as $\lambda \to +\infty$.
\end{corollary}

\subsection{Lower bound for the magnetic Laplacian near the peak}
    In this subsection, we set $I:= (b_0 \lambda^{-\kappa}, b_1\lambda^{- \frac{1}{2q}})$ for some $b_0, b_1 >0$ and $\kappa \in (\frac{1}{2q}, \frac{1}{q+1})$. The goal is to construct a lower bound for $Q_\lambda^I$, using the same technique that was used in the previous section.

\begin{lemma}
\label{lem: lower bound middle}
    For any $b_0, b_1 >0$ and $\kappa \in (\frac{1}{2q},\frac{1}{q+1})$ there exists $C >0$ such that
    \[
    L_\lambda^I \geq C\lambda^{2(1-q \kappa)}
    \]
    holds for sufficiently large $\lambda$ as $\lambda \to +\infty$.
\end{lemma}
\begin{proof}
    Let $b_1,b_2 >0$ and $ \kappa \in (\frac{1}{2q}, \frac{1}{q+1})$.
    Let $v$ be a normalised eigenfunction of $Q_\lambda^I$ associated with the first eigenvalue of $Q_\lambda^I$ and $P_0: L^2(\Pi_I) \to L^2(\Pi_I)$ be the projection defined in (\ref{eq: Pi 0}). As before, we write $v = P_0v + \tilde v $ and associate $P_0 v$ with its corresponding function in $L^2(I)$ depending on the given norm. It holds
    \begin{equation}
    \label{eq: idk anymore}
    E_1(L_\lambda^I ) \|v\|^2_{L^2(\Pi_I)} 
    = \ell_\lambda^I(v) = \ell_\lambda^I (P_0 v + \tilde v)
    \geq \ell_\lambda^I(P_0 v) - 2 |\ell_\lambda^I(P_0 v, \tilde v)|
    \end{equation}
    Due to Lemma \ref{lem: model operator middle piece} and Lemma \ref{lem: upper bound corner}, we already have a suitable lower bound for $\ell_\lambda^I (P_0v)$: Because of the calculations done in Lemma \ref{lem: lower bound corner} we know that $P_0 v \in H^1_0(I) $ and there exists $C>0$ such that it holds:
    \begin{equation}
    \label{eq: idk anymore 2}
    \ell_\lambda^I(P_0v) = (a_+-a_-) \tilde t^I_\lambda(P_0 v) \geq  (a_+-a_-) E_1(T_\lambda^I) \|P_0 v\|^2_{L^2(I)} = C \lambda^{2(1-q\kappa)} \|P_0v\|^2_{L^2(\Pi_I)}.
    \end{equation}
    Thus, we only need to estimate $\ell_\lambda^I (P_0 v, \tilde v)$. Analogous to the proof of Lemma \ref{lem: lower bound corner}, we  write $\ell_\lambda^I(P_0 v,\tilde v)$ as
    \begin{align*}
    \int\limits_I \int\limits_{a_-}^{a_+}  \left( i \lambda s^q t \tilde v -\frac{qt}{s} \partial_t \tilde v\right) \left(\partial_s \overline{P_0 v} - \frac{q}{2s} \overline{P_0 v} - i\lambda s^q t\overline{P_0 v}\right) + i \lambda s^q t \partial_s \overline{P_0 v} \tilde v  + i \frac{3}{2} \lambda qs^{q-1} t \overline{P_0 v} \tilde v   \, \ddd t \ddd s.
    \end{align*}
    And with Cauchy-Schwarz, we have
    \begin{align}
    \label{eq: Cauchy Schwarz 2}
    \begin{split}
    |\ell_\lambda^I(P_0v,\tilde v)| 
    &\leq \left(\sqrt{\lambda} b_1^q a_{max}\|\tilde v \|_{L^2(\Pi_I)} +\left\|\frac{qt}{s } \partial_t \tilde v \right\|_{L^2(\Pi_I)}\right) \left\|\partial_s P_0v +(i\lambda s^q t- \frac{q}{2s}) P_0 v  \right\|_{L^2(\Pi_I)} \\ 
    &+ \sqrt{\lambda } b_1^q a_{max} \| \partial_s P_0 v \|_{L^2(\Pi_I)} \| \tilde v \|_{L^2(\Pi_I)} + \lambda^{1-\frac{q-1}{2q}} \frac{3}{2}qb_1^{q-1} a_{max} \| P_0 v\|_{L^2(\Pi_I)} \| \tilde v\|_{L^2(\Pi_I)}.
    \end{split}
    \end{align}
    One can show, analogous to (\ref{eq: upper bound tilde v}), (\ref{eq: qt:s partial t v upper bound}), and (\ref{eq: Clambda =}) that there exists $c_0 >0$ such that the following inequalities hold:

    \begin{align}
    \label{eq: v tilde leq}
    \begin{split}
        \| \tilde v \|_{L^2(\Pi_I)} 
        &\leq E_2(-\Delta^N_{a_-,a_+})^{-\frac{1}{2}} \| \partial_t \tilde v \|_{L^2(\Pi_I)}
        = \frac{b_1^q\lambda^{-\frac{1}{2}}}{\sqrt{E_2(-\Delta_{a_+,a_-})}} \| (b_1\lambda^{-\frac{1}{2q}})^{- q} \partial_t v \|_{L^2(\Pi_I)}\\
        &\leq  \frac{b_1^q\lambda^{-\frac{1}{2}}}{\sqrt{E_2(-\Delta_{a_+,a_-})}} \| s^{-q} \partial_t v \|_{L^2(\Pi_I)}
        \leq  \frac{b_1^q\lambda^{-\frac{1}{2}}}{\sqrt{E_2(-\Delta_{a_+,a_-})}} \sqrt{\ell_\lambda^I(v)}
        \leq c_0 \lambda^{-\frac{1}{2}} \sqrt{E_1(L_\lambda^I)} \|v \|_{L^2(\Pi_I)}
    \end{split}
    \end{align}
    \begin{equation}
    \label{eq: qts partialt tilde v leq}
    \left\|\frac{qt}{s} \partial_t \tilde v \right\|_{L^2(\Pi_I)} 
    \leq  q a_{max}\left\|s^{-1} \partial_t  v \right\|_{L^2(\Pi_I)}
    \leq  q a_{max} b_1^{q-1} \lambda^{-\frac{q-1}{2q}} \sqrt{\ell_\lambda^I(v)}
    \leq c_0 \lambda^{-\frac{q-1}{2q}} \sqrt{E_1(L_\lambda^I)} \|v \|_{L^2(\Pi_I)} ,
    \end{equation}
    \begin{equation}
    \label{eq: Clambda=2}
    \left\| \partial_s P_0 v + (i \lambda s^q t - \frac{q}{2s}) P_0 v \right\|_{L^2(\Pi_I)} = \sqrt{\ell_\lambda^I(P_0v)}.
    \end{equation}
    We can also strengthen the estimate for $\|\partial_s P_0 v\|_{L^2(\Pi_I)}$. Using the identity 
    \[ 
    \langle a+b,a+b \rangle \geq |a|^2+|b|^2-2|\langle a,b\rangle| \geq |a|^2 + |b|^2 - 2(\delta|a|)(\frac{|b|}{\delta}) \geq |a|^2(1-\delta^2) + |b|^2(1-\frac{1}{\delta^2}),
    \]
    for some $\delta >0$
    there exists $c_1 >0$ such that
    \begin{align*}
        \tilde t^I_\lambda (P_0 v) 
        \geq \int\limits_{I}(1-\delta^2) |\partial_s(P_0 v)|^2 + \lambda^2 s^{2q}|P_0v |^2 \left(\frac{(a_+-a_-)^2}{12} + (1-\frac{1}{\delta^2})\frac{(a_++a_-)^2}{4} \right) -\frac{c_1}{s^2} |P_0 v|^2 \ddd s    .
    \end{align*}
    We choose $\delta \in (0,1)$ sufficiently close to $1$ such that it holds
    \begin{equation*}
    \left(\frac{(a_+-a_-)^2}{12} + (1-\frac{1}{\delta^2})\frac{(a_++a_-)^2}{4} \right) \geq 0,
    \end{equation*}
    then it follows that one has for some $c_2>0$
    \begin{equation}
    \label{eq: I hate you 2}
    \| \partial_s P_0 v\|_{L^2(\Pi_I)} 
    \leq c_2 \sqrt{\ell_\lambda^I (P_0 v)} + c_2 \lambda^\kappa \| P_0 v \|_{L^2(\Pi_I)}.
    \end{equation}
    
    Substituting (\ref{eq: v tilde leq}), (\ref{eq: qts partialt tilde v leq}), (\ref{eq: Clambda=2}), and (\ref{eq: I hate you 2}) into (\ref{eq: Cauchy Schwarz 2}) we get
    \begin{align*}
    |\ell_\lambda^I(P_0 v, \tilde v)| 
    &\leq \sqrt{\ell_\lambda^I(P_0v)} \sqrt{E_1(L_\lambda^I)} \|v \|_{L^2(\Pi_I)} \left( c_0b_1^qa_{max}+ c_0\lambda^{-\frac{q-1}{2q} } + c_0c_2b_1^qa_{max}  \right) \\
    &+ \left(c_2b_1^qa_{max}\lambda^{\frac{1}{2}+\kappa} + \frac{3}{2}qb_1^{q-1}\lambda^{1-\frac{q-1}{2q}}\right)\|P_0v \|_{L^2(\Pi_I)} \| \tilde v \|_{L^2(\Pi_I)} .
    \end{align*}
    And by implementing (\ref{eq: idk anymore 2}) and (\ref{eq: v tilde leq}) we get for some $c_3>0$:
    \[
    |\ell_\lambda^I (P_0 v, \tilde v)| \leq \sqrt{\ell_\lambda^I(P_0v)} \sqrt{E_1(L_\lambda^I)} \|v \|_{L^2(\Pi_I)} c_3\left(1+ \lambda^{-\frac{q-1}{2q}} +\lambda^{-1+\kappa(q+1)} + \lambda^{q\kappa -\frac{2q-1}{2q}} \right).
    \]
    Since $-\frac{q-1}{2q} <0$, $-1+\kappa(q+1) <0 $ and $ q\kappa - \frac{2q-1}{2q} <0$ hold for any $\kappa \in (\frac{1}{2q}, \frac{1}{q+1})$, we can combine the equation above with (\ref{eq: idk anymore}) and (\ref{eq: idk anymore 2}) to get for sufficiently large $\lambda$
    \[
    \ell_\lambda^I (P_0 v) \leq E_1(L_\lambda^I) \|v \|^2_{L^2(\Pi_I)} + \sqrt{E_1(L_\lambda^I)} \| v \|_{L^2(\Pi_I)} 3c_3 \sqrt{\ell_\lambda^I(P_0 v)}.
    \]
    Using (\ref{eq: abx}) with $x=\sqrt{\ell_\lambda^I(P_0 v)}$, we arrive at
    \[
     E_1(L_\lambda^I) \|v \|^2_{L^2(\Pi_I)}\bigl(1+3c_3+9c_3^2\bigr) \geq \ell_\lambda^I(P_0 v) .
    \]
    Combining Lemma \ref{lem: model operator middle piece} with (\ref{eq: v tilde leq}) we get for some $c_4 >0$
    \[
    \| \tilde v \|^2_{L^2(\Pi_I)} \leq c_0^2 \lambda^{-1} E_1(L_\lambda^I) \| v \|^2_{L^2(\Pi_I)} \leq c_0^2 \lambda^{-1} E_1(T_\lambda^I) \| v\|^2_{L^2(\Pi_I)} \leq c_4 \lambda^{1-2q\kappa} \| v \|^2_{L^2(\Pi_I)},
    \]
    which allows us to conclude using (\ref{eq: idk anymore 2}) and the fact that $1-2q\kappa<0$ holds:
    \begin{align*}
    \bigl(1+3c_3+9c_3^2\bigr)E_1(L_\lambda^I ) \| v\|^2_{L^2(\Pi_I)} 
    &\geq \ell_\lambda^I (P_0 v)
    \geq C \lambda^{2(1-q\kappa)} \| P_0v \|^2_{L^2(\Pi_I)}\\
    &\geq C \lambda^{2(1-q\kappa)}\|v\|^2_{L^2(\Pi_I)}\bigl(1-c_4\lambda^{1-2q\kappa}\bigr) . \qedhere
    \end{align*}

\end{proof}

\subsection{Lower bound for the magnetic Laplacian away from the peak}
In this subsection, we set $I:= (b_1 \lambda^{-\frac{1}{2q}} , b)$ for some $b,b_1>0$. The goal is to construct a lower bound for $Q_\lambda^I$, which is achieved through an IMS partition allowing us to localise near the Neumann boundary of $V_I$.
\begin{notation}
    We start by constructing a one-sided tubular neighbourhood of the curved Neumann boundary of $V_I$ in $V_I$. Consider the curves 
    \[\Gamma_\pm(s) = (s,a_\pm s^q), \quad s \in (0,2b)\]  
    and their arc-length parametrizations $ \gamma_+,\gamma_- $. Set $\tau_\pm$ as their unit tangential vectors,\\
    $\nu_\pm = \begin{pmatrix}
        0&-1\\1&0
    \end{pmatrix}\tau_\pm$ 
    as their unit normal vectors, and their curvatures are denoted by
    \[
    k_\pm(s) := \langle \tau_\pm',\nu_\pm \rangle = \frac{a_\pm q(q-1)s^{q-2}}{(1+a_\pm^2q^2 s^{2q-2})^{\frac{3}{2}}}.
    \]
    We also define the rectangle
    \[ \Pi_\lambda := (\frac{b_1}{2}\lambda^{-\frac{1}{2q}},2b) \times (0,C_0\lambda^{-\frac{1}{2}}), 
    \quad C_0 = \sqrt{1+q^2 a_{max}^2 (2b)^{2q-2}} \frac{1}{2} (a_+-a_-) b_1^q
    \]
    and the functions
    \[ 
    \phi_\pm(s,t) := \gamma_\pm(s) \mp t \nu_\pm(s), \quad (s,t) \in \Pi_\lambda.
    \]
    Then the matrix $G_\pm$ with $  (G_\pm)_{ij} = \langle \partial_i \phi_\pm,\partial_j \phi_\pm \rangle$ takes the form 
    \[
    G_\pm =
    \begin{pmatrix}
        (1-tk_\pm(s))^2 & 0 \\0&1
    \end{pmatrix}.
    \]
    Thus, there exists $c_0>0$ such that for all $(s,t) \in \Pi_\lambda$ it holds
    \begin{equation} 
    \label{eq: tk leq}
    |tk_\pm(s)| \leq ta_{max} q(q-1) s^{q-2} \leq \begin{cases}
        c_0 \lambda^{-\frac{q-1}{q}}, & q \in (1,2)
        \\ c_0 \lambda^{-\frac{1}{2}}, & q \geq 2
    \end{cases} \quad =: g(\lambda)  \to 0 \quad \text{as }\lambda \to +\infty.
    \end{equation}
    Therefore, the matrix $G_\pm$ is invertible for sufficiently large $\lambda$, and because of the tubular neighbourhood theorem, we know that there exists $\lambda_0$ such that for all $\lambda \geq \lambda_0$ the functions $\phi_\pm$ define diffeomorphisms between  $\Pi_\lambda$ and the one-sided tubular neighourhoods $ S^\pm_\lambda:= \phi_\pm(\Pi_\lambda)$ of sections of $\Gamma_\pm$. A visualisation can be found on Figure \ref{fig: S lambda pm}.
\end{notation}

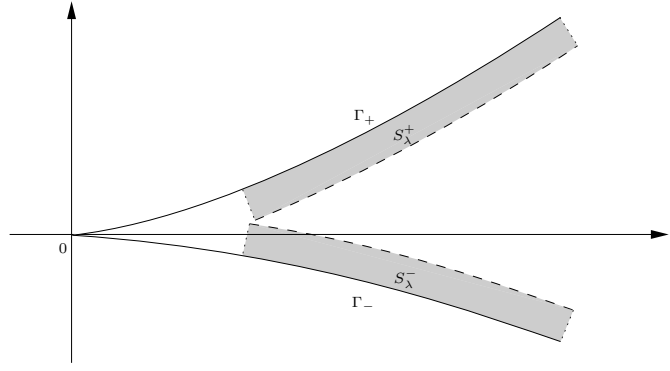
\begin{figure}[t]
    \centering
    \scalebox{0.8}{
    \begin{tikzpicture}[x=1pt,y=1pt,yscale=-1,xscale=1]

\draw    (240,256.02) -- (240,88.02) ;
\draw [shift={(240,86.02)}, rotate = 90] [fill={rgb, 255:red, 0; green, 0; blue, 0 }  ][line width=0.08]  [draw opacity=0] (8.4,-2.1) -- (0,0) -- (8.4,2.1) -- cycle    ;
\draw  [dash pattern={on 0.84pt off 2.51pt}]  (320.48,174.57) -- (326.29,188.95) ;
\draw  [dash pattern={on 0.84pt off 2.51pt}]  (469.6,93.62) -- (478,106.95) ;
\draw  [dash pattern={on 0.84pt off 2.51pt}]  (323.71,190.67) -- (320.29,206.1) ;
\draw  [dash pattern={on 0.84pt off 2.51pt}]  (476,231.24) -- (470,246.02) ;
\draw [fill={rgb, 255:red, 155; green, 155; blue, 155 }  ,fill opacity=0.5 ] [dash pattern={on 4.5pt off 4.5pt}]  (326.29,188.95) .. controls (380.29,166.67) and (435.71,134.95) .. (478,106.95) ;
\draw [fill={rgb, 255:red, 155; green, 155; blue, 155 }  ,fill opacity=0.5 ] [dash pattern={on 4.5pt off 4.5pt}]  (323.71,190.67) .. controls (393.43,200.95) and (442.57,219.81) .. (476,231.24) ;
\draw  [draw opacity=0][fill={rgb, 255:red, 155; green, 155; blue, 155 }  ,fill opacity=0.5 ] (476,231.24) -- (470,246.02) -- (320.29,206.1) -- (323.71,190.67) -- cycle ;
\draw [fill={rgb, 255:red, 255; green, 255; blue, 255 }  ,fill opacity=1 ]   (240.4,196.02) .. controls (355.2,203.62) and (439.14,235.81) .. (470,246.02) ;
\draw    (210.86,195.67) -- (518.86,195.67) ;
\draw [shift={(520.86,195.67)}, rotate = 180] [fill={rgb, 255:red, 0; green, 0; blue, 0 }  ][line width=0.08]  [draw opacity=0] (8.4,-2.1) -- (0,0) -- (8.4,2.1) -- cycle    ;
\draw  [draw opacity=0][fill={rgb, 255:red, 155; green, 155; blue, 155 }  ,fill opacity=0.5 ] (469.6,93.62) -- (478,106.95) -- (326.29,188.95) -- (320.48,174.57) -- cycle ;
\draw [fill={rgb, 255:red, 255; green, 255; blue, 255 }  ,fill opacity=1 ]   (240.4,196.02) .. controls (324.4,186.42) and (427.2,121.62) .. (470,93.62) ;

\draw (233,199.02) node [anchor=north west][inner sep=0.75pt]  [font=\tiny] [align=left] {$\displaystyle 0$};
\draw (390.1,143.33) node [anchor=north west][inner sep=0.75pt]  [font=\tiny] [align=left] {$\displaystyle S_{\lambda }^{+}$};
\draw (390.67,210.76) node [anchor=north west][inner sep=0.75pt]  [font=\tiny] [align=left] {$\displaystyle S_{\lambda }^{-}$};
\draw (371.73,136.33) node [anchor=north west][inner sep=0.75pt]  [font=\tiny] [align=left] {$\displaystyle \Gamma _{+}$};
\draw (370.93,223.93) node [anchor=north west][inner sep=0.75pt]  [font=\tiny] [align=left] {$\displaystyle \Gamma _{-}$};

\end{tikzpicture}
    }
    \caption{ $S_\lambda^\pm$, the one-sided tubular neighborhoods of $\Gamma_\pm$}
    \label{fig: S lambda pm}
\end{figure}
    This allows us to introduce the following operator:
    
\begin{definition}
    We will consider the magnetic Laplacian on $S^\pm_\lambda$ with Neumann boundary conditions on $\Gamma_\pm$ and Dirichlet boundary conditions on the rest of the boundary. 
    To be more precise, we set
    \begin{align*}
        D^\pm_\lambda =  \Bigl\{  f \in H^1(S^\pm_\lambda) \mid &f = 0 \text{ on }  \phi_\pm \left( (\frac{b_1}{2}\lambda^{-\frac{1}{2q}},2b) \times \{ C_0 \lambda^{-\frac{1}{2}} \}\right), \\
        & f=0  \text{ on } \phi_\pm \left( \{\frac{b_1}{2}\lambda^{-\frac{1}{2q}},2b \} \times (0,C_0 \lambda^{-\frac{1}{2}})\right) \Bigr\}
    \end{align*}
    and denote by $Q^\pm_\lambda$ the self-adjoint operators in $L^2$ generated by the quadratic forms
    $q^\pm_\lambda (f) = n_\lambda^{S_\lambda^\pm}(f) $ with domain $ 
    D(q_\lambda^\pm) =D_\lambda^\pm $.
\end{definition}

The following lower bound can be shown for this magnetic Laplacian:

\begin{lemma}
\label{lem: lower bound on stripes}
For any $b,b_1>0$ it holds 
\[ Q_{\lambda}^{\pm} \geq \frac{\Theta_0}{2} \lambda \]
as $\lambda \to + \infty$.
\end{lemma}
\begin{proof}
    Let $b,b_1 >0$.
    By utilizing the diffeomorphisms $\phi_\pm$ and a gauge transformation, it is shown in \cite[Appendix B]{Forunais_2006} that there exists a unitary operator 
    \[
     \Phi_\pm:L^2(S_\lambda^\pm) \to L^2(\Pi_\lambda,1-tk_\pm), \quad \Phi_\pm(f) =: u,
    \]
    such that 
    \begin{align*} 
    &q_\lambda^\pm (f) = \tilde q_\lambda^\pm(u)
    := \int \limits_{\Pi_\lambda} (1-tk_\pm(s))^{-1}|\partial_s u - i \lambda (-t + \frac{t^2 k_\pm(s)}{2})u |^2 + (1-tk_\pm(s))|\partial_t u|^2 \ddd s \ddd t ,\\
    &D(\tilde q_\lambda^\pm) = \{ u \in H^1(\Pi_\lambda,1-tk_\pm(s) ) \mid u(0,t) = u (2b,t)=u(s,C_0 \lambda^{-\frac{1}{2}}) = 0 \ \forall \ t\in(0,C_0\lambda^{-\frac{1}{2}}),\ s \in (0,2b) \}.
    \end{align*}
    It is apparent that $\tilde Q_\lambda^\pm$ is unitary equivalent to $Q_\lambda^\pm$ and we can estimate $\tilde q_\lambda^\pm$ by implementing (\ref{eq: tk leq}) and the inequality
    $ |x+y|^2 \geq (1-\varepsilon)|x|^2 - \frac{1}{\varepsilon}|y|^2 $
    : As $|1-tk_\pm(s)| \to 1$ on $\Pi_\lambda$ as $\lambda \to + \infty$ there exists for any fixed $\varepsilon \in (0,1/4)$ a $\lambda_0>0$ such that for all $\lambda \geq \lambda_0$ it holds
    \begin{align*}
    \tilde q_\lambda^\pm(u)
    & \geq (1-\varepsilon)\int\limits_{\Pi_\lambda}|\partial_s u - i \lambda (-t + \frac{t^2k_\pm(s)}{2})u|^2 + |\partial_t u| \ddd t \ddd s \\
    &\geq (1-\varepsilon) \int \limits_{\Pi_\lambda} |\partial_t u|^2 + (1-\varepsilon) |\partial_su +i\lambda t u|^2 -\frac{1}{\varepsilon} |t^ 2 k_\pm(s) u|^2 \ddd t \ddd s   \\
    &\geq (1-\varepsilon)^2 \int \limits_{\Pi_\lambda} |\partial_t u|^2 + |\partial_s u + i \lambda t u|^2 - |u|^2 \frac{C_0}{\varepsilon} \lambda^{-\frac{1}{2}}g(\lambda) \ddd t \ddd s  \\
    &\geq (1-\varepsilon)^2 \int \limits_{\R^2_+} |\partial_t \tilde u|^2 + |\partial_s \tilde u + i \lambda t \tilde u|^2 - |\tilde u|^2 \frac{C_0}{\varepsilon}\lambda^{-\frac{1}{2}}g(\lambda) \ddd t\ddd s,
    \end{align*}
    where $\tilde u$ is defined as the continuation of $u$ with zero onto $\R^2_+ $. Since $1-tk_\pm \in(1-\varepsilon,1+\varepsilon) $ on $\Pi_\lambda$, one has $u \in H^1(\Pi_\lambda)$ and, therefore $\tilde u \in H^1(\R^2_+) =D( \tilde n_\lambda^{\R^2_+})$. Due to Lemma \ref{lem: 2-dim model op}, we know that $\tilde N^{\R^2_+}_\lambda$ is bounded from below by $ \Theta_0 \lambda $ and we get for $\lambda $ large enough
    \[
    Q_\lambda^{\pm} \geq  (1-\varepsilon)^2 \left( \lambda \Theta_0 - \frac{C_0}{\varepsilon}\lambda^{-\frac{1}{2}}g(\lambda) \right) \geq \frac{\Theta_0}{2}\lambda. \qedhere
    \]
\end{proof}

This enables us to prove the lower bound for $Q_\lambda^I$ where $I = (b_1\lambda^ {-\frac{1}{2q}},b)$.

\begin{lemma}
\label{lem: lower bound away from corner} 
    For any $b>0$ there exist $b_1,\lambda_0,C>0$ such that for all $\lambda \geq \lambda_0$ one can set $I:= (b_1 \lambda^{-\frac{1}{2q}}, b)$ and it holds
    $$ Q_\lambda^{I} \geq C \lambda $$
    as $\lambda \to + \infty$.
\end{lemma}
\begin{proof}
    Let $b>0$.
    Pick smooth functions 
    \[ 
    \chi,\tilde \chi: \R \rightarrow [0,1], \quad \chi(t) = \begin{cases}
        1, & t \leq \frac{1}{2} \\ 0, & t \geq \frac{3}{4}
    \end{cases}, \quad \chi^2+\tilde\chi^2 \equiv 1, \quad \chi_\infty:= \|\chi' \|_\infty^2  + \|\tilde \chi' \|_\infty^2
    \]
    and define with them the following cutoff functions on $V_I$:
    \[
    \chi_\lambda^\pm: V_I \to [0,1],\quad 
    \chi_\lambda^\pm(x,y) := \chi \left(\mp\left(y-a_\pm x^q\right) 2(a_+ -a_-)^{-1} b_1^{-q} \sqrt\lambda \right).
    \]
    Since one has
    \begin{align*}
    &\supp \chi_\lambda^+ \subsetneq \{ (x,y) \in V_I \mid  y > a_+x^q -\frac{a_+-a_-}{2}\lambda^{-\frac{1}{2}} b_1^q \}=: U_\lambda^+,\\
    &\supp \chi_\lambda^- \subsetneq \{ (x,y) \in V_I \mid  y < a_-x^q +\frac{a_+-a_-}{2}\lambda^{-\frac{1}{2}} b_1^q \} =: U_\lambda^-,
    \end{align*}
    the function 
    \[
    \chi_\lambda^0: V_I \to [0,1], \quad 
    \chi_\lambda^0(x,y) = \begin{cases}
         \tilde \chi \left(\left(y-a_+ x^q\right) 2(a_- -a_+)^{-1} b_1^{-q} \sqrt\lambda \right), & y \geq \frac{a_+ +a_-}{2}x^q \\
         \tilde\chi \left(\left(y-a_- x^q\right) 2(a_+ -a_-)^{-1} b_1^{-q} \sqrt\lambda \right),  & y < \frac{a_+ +a_-}{2}x^q
    \end{cases}
    \]
    is $C^\infty$ and it holds $(\chi_\lambda^+)^2 +(\chi_\lambda^-)^2+(\chi_\lambda^0)^2\equiv 1$. A visualization of the domains $U_\lambda^\pm$ that contain the supports of $\chi_\lambda^\pm$ can be found on Figure \ref{fig: U lambda pm} (A).
    \begin{figure}[t]
        \centering
        \begin{subfigure}[b]{0.45\textwidth}
        \scalebox{0.7}{
            \begin{tikzpicture}[x=0.65pt,y=0.65pt,yscale=-1,xscale=1]

\draw    (129.83,246.08) -- (130.33,50.58) ;
\draw [shift={(130.33,48.58)}, rotate = 90.15] [fill={rgb, 255:red, 0; green, 0; blue, 0 }  ][line width=0.08]  [draw opacity=0] (8.4,-2.1) -- (0,0) -- (8.4,2.1) -- cycle    ;
\draw    (111.33,180.08) -- (518.83,180.08) ;
\draw [shift={(520.83,180.08)}, rotate = 180] [fill={rgb, 255:red, 0; green, 0; blue, 0 }  ][line width=0.08]  [draw opacity=0] (8.4,-2.1) -- (0,0) -- (8.4,2.1) -- cycle    ;
\draw    (268.33,138.83) -- (267.83,201.58) ;
\draw  [draw opacity=0][fill={rgb, 255:red, 155; green, 155; blue, 155 }  ,fill opacity=0.5 ] (268.33,138.83) -- (404.33,57.58) -- (404.33,241.08) -- (267.83,201.58) -- cycle ;
\draw [fill={rgb, 255:red, 255; green, 255; blue, 255 }  ,fill opacity=1 ]   (268.33,138.83) .. controls (311.33,117.83) and (363.83,85.08) .. (404.33,57.58) ;
\draw [fill={rgb, 255:red, 255; green, 255; blue, 255 }  ,fill opacity=1 ]   (267.83,201.58) .. controls (324.33,215.08) and (370.33,230.08) .. (404.33,241.08) ;
\draw  [draw opacity=0][fill={rgb, 255:red, 255; green, 255; blue, 255 }  ,fill opacity=1 ] (404.08,88.96) -- (404.58,209.71) -- (268.08,170.21) -- cycle ;
\draw [fill={rgb, 255:red, 155; green, 155; blue, 155 }  ,fill opacity=0.5 ]   (268.08,170.21) .. controls (276.59,172.24) and (284.87,174.31) .. (292.89,176.39) .. controls (338.17,188.14) and (375.7,200.37) .. (404.58,209.71) ;
\draw    (404.33,57.58) -- (404.33,241.08) ;
\draw    (111.33,180.08) -- (518.83,180.08) ;
\draw [shift={(520.83,180.08)}, rotate = 180] [fill={rgb, 255:red, 0; green, 0; blue, 0 }  ][line width=0.08]  [draw opacity=0] (8.4,-2.1) -- (0,0) -- (8.4,2.1) -- cycle    ;
\draw [fill={rgb, 255:red, 155; green, 155; blue, 155 }  ,fill opacity=0.5 ]   (268.08,170.21) .. controls (311.08,149.21) and (363.58,116.46) .. (404.08,88.96) ;

\draw (118,184.5) node [anchor=north west][inner sep=0.75pt]  [font=\scriptsize] [align=left] {$\displaystyle 0$};
\draw (323,112) node [anchor=north west][inner sep=0.75pt]  [font=\scriptsize] [align=left] {$\displaystyle U_{\lambda }^{+}$};
\draw (319.67,195.33) node [anchor=north west][inner sep=0.75pt]  [font=\scriptsize] [align=left] {$\displaystyle U_{\lambda }^{-}$};

\end{tikzpicture}
        }
        \caption{$U_\lambda^\pm$ containing the support of $ \chi_\lambda^\pm $}
        \end{subfigure}
        \begin{subfigure}[b]{0.45\textwidth}
         \scalebox{0.7}{
         
\begin{tikzpicture}[x=0.95pt,y=0.95pt,yscale=-1,xscale=1]

\draw  [dash pattern={on 4.5pt off 4.5pt}]  (250,160) -- (254.83,196.92) ;
\draw  [dash pattern={on 4.5pt off 4.5pt}]  (420,90) -- (430.82,104.88) -- (442.83,121.42) ;
\draw [fill={rgb, 255:red, 155; green, 155; blue, 155 }  ,fill opacity=0.5 ] [dash pattern={on 0.84pt off 2.51pt}]  (254.83,196.92) .. controls (325.33,188.42) and (403.83,151.42) .. (442.83,121.42) ;
\draw  [draw opacity=0][fill={rgb, 255:red, 155; green, 155; blue, 155 }  ,fill opacity=0.5 ] (420,90) -- (442.83,121.42) -- (254.83,196.92) -- (249.87,160.67) -- cycle ;
\draw [fill={rgb, 255:red, 255; green, 255; blue, 255 }  ,fill opacity=1 ]   (250,160) .. controls (320.83,150.42) and (380,120) .. (420,90) ;

\draw (312.53,115.73) node [anchor=north west][inner sep=0.75pt]   [align=left] {$\displaystyle \Gamma _{+}$};
\draw (317.73,148.73) node [anchor=north west][inner sep=0.75pt]   [align=left] {$\displaystyle S_{\lambda }^{+}$};
\draw (216.93,176.33) node [anchor=north west][inner sep=0.75pt]  [font=\scriptsize] [align=left] {$\displaystyle \partial _{1} S_{\lambda }^{+}$};
\draw (434.53,83.93) node [anchor=north west][inner sep=0.75pt]  [font=\scriptsize] [align=left] {$\displaystyle \partial _{2} S_{\lambda }^{+}$};
\draw (356.13,181.13) node [anchor=north west][inner sep=0.75pt]  [font=\scriptsize] [align=left] {$\displaystyle \partial _{*} S_{\lambda }^{+}$};

\end{tikzpicture}

         }
         \caption{ $S_\lambda^+$ and its boundary components }
        \end{subfigure}
        \caption{Visualizations of $U_\lambda^\pm$ and $S_\lambda^\pm$}
        \label{fig: U lambda pm}
    \end{figure}
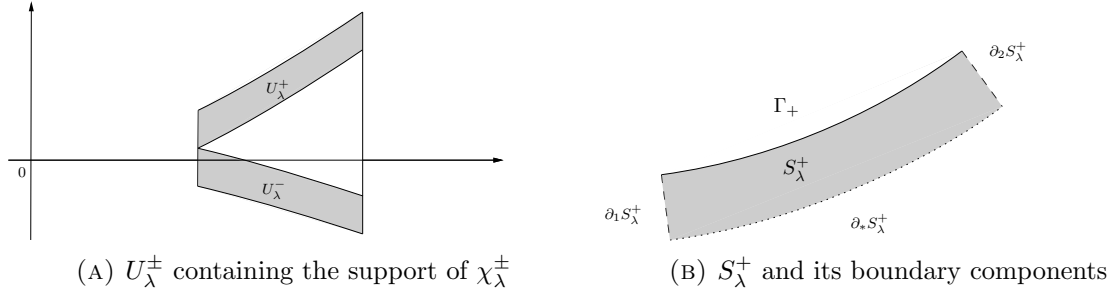
    This allows us to do an IMS partition. 
    Set
    \[
    J: H^1_I(V_I) \to D(q_\lambda^+) \oplus D(q_\lambda^-) \oplus H^1_0(V_{(0,b)}), \quad J(f) = (\chi_\lambda^+ f,\chi_\lambda^-f,\chi_\lambda^0 f),
    \]
    where we extend $\chi_\lambda^0f$ with $0$ onto $V_{(0,b)}$ and  $\chi_\lambda^\pm f$ with $0$ onto $S_\lambda^\pm$ respectively and get with Lemma \ref{lem: IMS partition}
    \[
    q_\lambda^I(f) \geq 
    (q_\lambda^+ \oplus q_\lambda^-\oplus d_\lambda^{V_{(0,b)}}) (J f) - c_0\lambda \| f \|^2_{L^2(V_I)},
    \]
    where $D_\lambda^{V_{(0,b)}}$ is the magnetic Dirichlet Laplacian on $V_{(0,b)}$ and $c_0$ is defined as
    \[
    c_0 := 8 \left(1+a^2_{max}q^2b^{2(q-1)}\right) (a_+ - a_-)^{-2} \chi_\infty b_1^{-2q}\geq  \bigl(\|\nabla \chi_\lambda^+ \|^2_\infty + \|\nabla \chi_\lambda^- \|^2_\infty + \| \nabla \chi_\lambda^0 \|_\infty ^2 \bigr),
    \]
    because we have

    \begin{align*}
    &|\partial_x \chi_\lambda^\pm|^2 
    \leq |2(a_+-a_-)^{-1} b_1^{-q} \sqrt{\lambda} a_\pm qx^{q-1} \|\chi' \|_\infty |^2
    \leq \lambda 4 (a_+-a_-)^{-2} b_1^{-2q} a_{max}^2q^2 b^{2(q-1)} \| \chi' \|_\infty^2, \\
    &|\partial_y \chi_\lambda^\pm|^2 
    \leq |2(a_+-a_-)^{-1} b_1^{-q} \sqrt{\lambda}   \|\chi' \|_\infty |^2
    \leq \lambda 4 (a_+-a_-)^{-2} b_1^{-2q}   \| \chi' \|_\infty^2.
    \end{align*}
    If we can show that $\chi_\lambda^\pm f \in D(q_\lambda^\pm)$ holds we get with Lemma \ref{comparing operators}:
    \[
    E_1(Q_\lambda^I) + c_0 \lambda \geq E_1(Q_\lambda^+ \oplus Q_\lambda^- \oplus D_\lambda^{V_{(0,b)}}) .
    \]
    Since due to Lemma \ref{lem: lower bound dirichlet} one has $D^{V_{(0,b)}}_\lambda \geq  \lambda$ we can combine this with Lemma \ref{lem: lower bound on stripes} to get
    \[
    E_1(Q_\lambda^I) \geq \lambda\left(\frac{\Theta_0}{2} - c_0 \right),
    \]
    and we arrive at our conclusion by choosing $b_1$ sufficiently large such that $c_0 < \frac{\Theta_0}{2}$ holds.

    It remains to show that the continuation of $\chi_\lambda^\pm f$ with $0$ onto $S_\lambda^\pm$ lies in $D(q_\lambda^\pm)$.
    Note that per construction and due to the Dirichlet boundary conditions imposed on the straight parts of $\partial V_I$, it suffices to prove that $ U_\lambda^\pm  \subseteq S_\lambda^\pm $ holds.
    
    Let us start by computing the boundary of $S_\lambda^\pm$, which we split up as shown on Figure \ref{fig: U lambda pm} (B) for $S_\lambda^+$.
    As the non-normalised tangential and normal vectors of $\Gamma_\pm$ are
    \[
    T_\pm(s) = \begin{pmatrix}
        1\\a_\pm qs^{q-1}
    \end{pmatrix}, \quad N_\pm(s) = \begin{pmatrix}
        -a_\pm qs^{q-1}\\1
    \end{pmatrix},
    \]
    the part of the  boundary that is parallel to $\Gamma_\pm$ can be described by
    \[
    \partial_* S_\lambda^\pm :=
    \left\{\Gamma_\pm(s) \mp C_0 \lambda^{-\frac{1}{2}}  \frac{N_\pm(s)}{\|N_\pm(s) \|} = \begin{pmatrix}
        s \pm C_0 \lambda^{-\frac{1}{2}} \bigl(1+a_\pm^2 q^2 s^{2q-2}\bigr)^{-\frac{1}{2} } a_\pm qs^{q-1} 
        \\ a_\pm s^q \mp C_0 \lambda^{-\frac{1}{2}} \bigl(1+a_\pm^2q^2 s^{2q-2}\bigr)^{-\frac{1}{2}}
        \end{pmatrix} \mid s\in \left(\frac{b_1}{2} \lambda^{-\frac{1}{2q}},2b\right) \right\}.
    \]    
    Let $(x,y) \in \partial_*S_\lambda^\pm $ then there exists $s \in (b_1/2 \lambda^{-\frac{1}{2q}}, 2b)$ such that
    \[
    \begin{pmatrix}
        x \\y
    \end{pmatrix}
    =\begin{pmatrix}
        s \pm C_0 \lambda^{-\frac{1}{2}} \bigl(1+a_\pm^2 q^2 s^{2q-2}\bigr)^{-\frac{1}{2} } a_\pm qs^{q-1} 
        \\ a_\pm s^q \mp C_0 \lambda^{-\frac{1}{2}} \bigl(1+a_\pm^2q^2 s^{2q-2}\bigr)^{-\frac{1}{2}}
        \end{pmatrix}
    =: \begin{pmatrix}
        s+f_\pm(s) \\a_\pm s^q \mp C_0 \lambda^{-\frac{1}{2}} \bigl(1+a_\pm^2q^2 s^{2q-2}\bigr)^{-\frac{1}{2}}
    \end{pmatrix}.
    \]
    We want to show that this representation exists for all $x \in (b_1 \lambda^{-\frac{1}{2q}},b)$ and that it holds
    \begin{equation}
    \label{eq: how could this happen to me}
    y \leq a_+x^q - \frac{a_+-a_-}{2} \lambda^{-\frac{1}{2}} b_1^q, \quad \text{or respectively } \quad y \geq a_-x^q + \frac{a_+-a_-}{2} \lambda^{-\frac{1}{2}} b_1^1.
    \end{equation}
    
    Case 1: It holds $a_+ \geq0$ or $a_- \leq0$. \\
    Note that one has $x= s+f_\pm(s)$ and it holds for $\lambda$ sufficiently large:
    \begin{equation}
    \label{eq: f leq smth}
    f_\pm(s) = \pm a_\pm  q\bigl( s^{-2(q-1)} +a^2_\pm q^2 \bigr)^{-\frac{1}{2}} \lambda^{-\frac{1}{2}} < \frac{b_1}{2} \lambda^{-\frac{1}{2q}}, \quad \forall s \in (\frac{b_1}{2}\lambda^{-\frac{1}{2q}},2b ).
    \end{equation}
    Combining this with the fact that $f_\pm(s) \geq 0$ for all $s>0$ one has
    \[
    \frac{b_1}{2}\lambda^{-\frac{1}{2q}} +f_\pm(\frac{b_1}{2} \lambda) <b_1\lambda^{-\frac{1}{2q}}, \quad 2b +f_\pm(2b)>b,
    \]
    and due to $s+f_\pm(s)$ being continuous we know that any $x \in (b_1\lambda^{-\frac{1}{2q}},b)$ can be represented as $x=s+f_\pm(s)$. Furthermore, due to $ s = x-f_\pm(s) \leq x$ we have the following estimate for $y$:
    \begin{equation*}
    y =a_+ s^q-C_0 \lambda^{-\frac{1}{2}} \bigl(1+a_+^2q^2s^{2q-2}\bigr)^{-\frac{1}{2}}
    \leq a_+ x^q - C_0 \lambda^{-\frac{1}{2}} \bigl(1+ a_{max}^2q^2(2b)^{2q-2}\big)^{-\frac{1}{2}} = a_+ x^q -  \frac{a_+-a_-}{2}b_1^q \lambda^{-\frac{1}{2}}.
    \end{equation*}
    As the estimate can be done analogously for $S_\lambda^-$ we arrive at (\ref{eq: how could this happen to me}) for the first case.

   Case 2: It holds $a_+ <0$ or $a_- >0$. \\
   Here, we have $f_\pm(s) <0$ for all $s>0$, which we can combine with an  estimate analogue to (\ref{eq: f leq smth}): It holds $f_\pm(s)>-b_1/2 \lambda^{-1/(2q)}$ for any $s \in (b_1/2 \lambda^{-1/(2q)},2b) $ and $\lambda$ sufficiently large and therefore
   \[
    \frac{b_1}{2}\lambda^{-\frac{1}{2q}} +f_\pm(\frac{b_1}{2} \lambda) <b_1\lambda^{-\frac{1}{2q}}, \quad 2b +f_\pm(2b)>2b-\frac{b_1}{2}\lambda^{-\frac{1}{2q}} >b .
   \]
    The continuity of $s+f_\pm(s)$ implies the representation $x = s+f_\pm(s)$ for any $x \in (b_1\lambda^{-1/(2q)},b)$ again. As one has $s=x-f_\pm(s) >x$, the estimate (\ref{eq: how could this happen to me}) for $y$ works analogously, which concludes our study of $\partial_* S_\lambda^\pm$. 
    
    Concerning the straight parts of the boundary, they consist of
    \[
    \partial_1S_\lambda^\pm =
    \left\{ 
    \begin{pmatrix} \frac{b_1}{2}\lambda^{-\frac{1}{2q}} \pm t \bigl(1+a_\pm^2q^2( \frac{b_1}{2}\lambda^{-\frac{1}{2q}})^{2q-2}\bigr)^{-\frac{1}{2}} a_\pm \bigl( \frac{b_1}{2}\lambda^{-\frac{1}{2q}} \bigr)^{q-1} \\
    a_\pm \bigl( \frac{b_1}{2}\lambda^{-\frac{1}{2q}} \bigr)^q \mp t \bigl(1+a^2_\pm q^2 ( \frac{b_1}{2}\lambda^{-\frac{1}{2q}})^{2q-2}\bigr)
    \end{pmatrix} \mid t\in(0,C_0\lambda^{-\frac{1}{2}})
    \right\}
    \]
    and
    \[
    \partial_2S_\lambda^\pm =
    \left\{ 
    \begin{pmatrix} 2b \pm t \bigl(1+a_\pm^2q^2(2b)^{2q-2}\bigr)^{-\frac{1}{2}} a_\pm q(2b)^{q-1} \\
    a_\pm (2b)^q \mp t \bigl(1+a^2_\pm q^2 (2b)^{2q-2} \bigr)
    \end{pmatrix} \mid t\in(0,C_0\lambda^{-\frac{1}{2}})
    \right\}.
    \]
    So due to $|f_\pm(s)| \leq b_1/2 \lambda^{-\frac{1}{2q}}$ it is clear that for sufficiently large $\lambda$ it holds 
    \[
    (\partial_1 S_\lambda^\pm \cup \partial_2 S_\lambda^\pm )\cap ( b_1\lambda^{-\frac{1}{2q}},b ) = \emptyset. 
    \] 
    To summarise: $\phi_\pm: \Pi_\lambda \to S_\lambda^\pm $ is a diffeomorphism and $\partial S_\lambda^\pm \cap (b_1 \lambda^{-\frac{1}{2q}},b) \times \R$ consists of the curve $(s,a_\pm s^q)$ and $(x,y) \in  V_I$ laying on a curve and fulfilling the estimate (\ref{eq: how could this happen to me}). Thus,
    \[
    \supp \chi_\lambda^\pm \subset U_\lambda^\pm \subset S_\lambda^\pm \cap (b_1 \lambda^{-\frac{1}{2q}}, b) \subset S_\lambda^\pm
    \]
    which concludes the proof.
\end{proof}

\section{Proof of the main Theorem}
\label{sec: Endzeit}

Let us first establish a rigorous definition of the class of domains we consider.

\begin{definition}
     An open, bounded, connected set $\Omega \subseteq \R^2$ is called  a curvilinear polygon with an outward peak at the origin of sharpness order $q$ if 
     \begin{itemize}
     \item There exist $\ell_0,\dots,\ell_M>0$ and injective, arc-length parametrized smooth 
     curves $\gamma_j:[0,\ell_j] \to \R^2$ such that
     \begin{itemize}
         \item The interiors $\gamma(0,\ell_j), j = 1,\dots ,M$ are pairwise disjoint.
         \item It holds $\gamma_j(\ell_j)=\gamma_{j+1}(0)$ $\forall j =1,M-1$ and $\gamma_M(\ell_M) = \gamma_1(0)=0$.
         \item One has the decomposition $\partial\Omega = \bigcup\limits_{j=1}^M \gamma_j([0,l_j])$.
     \end{itemize}
     \item Furthermore, we assume that each $\gamma_j$ is oriented in such a way that, for the outward normal $\nu_j(s)$ to the bounded enclosed region by $\Gamma$ at a point $s\in(0,l_j)$, we have $\nu_j(s)\wedge \gamma'(s)=1$. Moreover, let $\alpha_j \in [0,2\pi]$, $j =2,\dots, M$ be the angle between the tangent vectors of $\gamma_{j-1}$ and $\gamma_j$, defined by the relations
    \[
        \cos (\alpha_j) = - \langle \nabla \gamma_j(0), \nabla \gamma_{j-1} (l_{j-1})  \rangle, \quad  \sin ( \alpha_j) = -\det( \nabla \gamma_j (0) \; \,\,\, \nabla \gamma_{j-1}(l_{j-1}) ).
\]
    We assume that $ \alpha_j\notin \{  0, \pi,2\pi \} $ for all $j = 2, \ldots, M$ and that the singularity at $0$ is exactly of the form that
        \item There exist $b>0$ and $a_+,a_- \in \R$, $a_- < a_+$ such that it holds
     \[
     \Omega \cap (-b,b)^2 = V_{(0,b)}.
     \]
    
     \end{itemize}
\end{definition}

From now on, we assume that $\Omega \subseteq \R^2 $ is a curvilinear polygon with an outward peak at the origin of sharpness order $q$ for some fixed $b>0$ and $a_-,a_+ \in \R$, $a_-< a_+$. 

Let us proceed by establishing an upper bound for the eigenvalues of $N_\lambda^\Omega$.

\begin{proposition}
\label{lem: upper bound Omega}
    Let $n \in \N$ and $\kappa \in (0, \frac{1}{q+1})$. Then there exists $C>0$ such that
    \[
    E_n(N_\lambda^\Omega) \leq \lambda^{\frac{2}{q+1}}E_n(T_1)+C\lambda^{2\kappa} 
    \]
    holds as $\lambda \to + \infty$.
\end{proposition}

\begin{proof}
    Due to our assumptions on $\Omega$, there exists $b>0$ such that
    $ \Omega \cap (-b,b)^2 = V_{(0,b)} $. Set $I:= (0, \lambda^{-\kappa})$  then one has $V_I \subset V_{(0,b)} \subset \Omega$ for $\lambda$ sufficiently large.
    Consider the inclusion
    $$ J: H^1_I(V_I) \hookrightarrow H^1(\Omega) \quad u \mapsto \tilde u :=\text{ the extension of }u \text{ with 0 onto } \Omega .$$
    Due to $q_\lambda^I(u) = n_\lambda^\Omega(Ju)$ we can  combine Corollary \ref{cor: upper bound corner} and Lemma \ref{comparing operators}: For any $\kappa \in (0, \frac{1}{q+1}) $ there exists $C>0$ such that:
    \begin{equation*}
    \lambda^{\frac{2}{q+1}} E_n(T_1) \bigl(1+C\lambda^{-2(\frac{1}{q+1}-\kappa)} \bigr)
    \geq E_n(T_\lambda^{I})
    \geq E_n(L^{I}_\lambda)
    = E_n(Q_\lambda^{{I}}) \geq E_n(N_\lambda^\Omega). \qedhere
    \end{equation*}
\end{proof}

It is left to show the existence of a suitable lower bound for the eigenvalues of $N_\lambda^\Omega$.

\begin{proposition}
\label{lem: lower bound Omega}
    Let $n \in \N$ and $\kappa \in (\frac{2q+1}{3q(q+1)}, \frac{1}{q+1})$. Then there exists $\lambda_0,C>0$ such that for all $\lambda  \geq \lambda_0$ one has
    \[
    E_n(N_\lambda^\Omega) \geq \lambda^{\frac{2}{q+1}} E_n(T_1) - C \lambda^{2\kappa}-C\lambda^{\frac{2q+3}{q+1}-3q\kappa}
    \]
    as $\lambda \to \infty$.
    
\end{proposition}

\begin{proof}
    Let $b_1>0$ and consider the $C^\infty$ cutoff functions
    \[
     \chi,\tilde \chi:\R \rightarrow \R_+, \quad \chi(s) = \begin{cases}
        1, & s \leq \frac{1}{2} \\ 0, & s \geq 1
    \end{cases}, \quad \chi^2+\tilde\chi^2 \equiv 1,
    \]
    then we construct a partition of unity $1 = \chi_1^2 + \chi_2^2 + \chi_3^2 +\chi_4^2 $ of $\Omega$ where $\chi_i$ are defined as follows:
    \[
    \chi_1(x,y) = \begin{cases}
        \chi ({x}{\lambda^{\kappa}}), & (x,y) \in V_{(0,b)} \\ 0, & \text{otherwise}
    \end{cases} 
    , \quad
     \chi_2(x,y) = \begin{cases}
        \tilde \chi(x\lambda^\kappa), & (x,y) \in V_{(0,b)}, \ x \leq 2\lambda^{-\kappa} \\ 
        \chi\bigl(xb_1^{-1}\lambda^{\frac{1}{2q}}\bigr), & (x,y ) \in V_{(0,b)} , \ x> 2\lambda^{-\kappa}\\
        0, & \text{otherwise} 
    \end{cases} \ ,
    \]
    \[
     \chi_3(x,y) = \begin{cases}
        \tilde \chi\bigl(x b_1^{-1}\lambda^\frac{1}{2q}\bigr),  & (x,y) \in V_{(0,b)} ,x \ \leq 2b_1\lambda^{-\frac{1}{2q}} \\ 
        \chi(xb^{-1}), & (x,y ) \in V_{(0,b)} , \ x> 2b_1\lambda^{-\frac{1}{2q}}\\
        0, & \text{otherwise} 
    \end{cases}, \quad
    \chi_4 (x,y) = \begin{cases}
        \tilde\chi(xb^{-1}), & (x,y) \in V_{(0,b)} \\ 1, & \text{otherwise}
    \end{cases} \ .
    \]
    For $\lambda$ sufficiently large we have $\chi_j \in C^\infty(\Omega)$, $j=1,\dots,4$. Set 
    \[
    I_1 := (0,\lambda^{-\kappa}), \quad I_2:= (\frac{\lambda^{-\kappa}}{2},b_1\lambda^{-\frac{1}{2q}}), \quad I_3 :=(\frac{b_1\lambda^{-\frac{1}{2q}}}{2},b), \quad \tilde\Omega := \Omega \setminus \overline{V_{(0,\frac b2)}},
    \]

    Note that $\tilde \Omega$ is a curvilinear polygon with at least two convex corners: There exists a decomposition into 
    smooth curves 
    \[
    \gamma_2, \dots, \gamma_{M-1}, \quad  \gamma_1 \setminus \partial V_{(0,b/2)} ,  \quad  \gamma_M \setminus \partial V_{(0,b/2)}, \quad \gamma_0:= \{(b/2,y) \mid y \in (a_- 2^{-q}b^q , a_+ 2^{-q} b^q \}
    \]
    and the angles of the corners are in $(0,2\pi)\setminus\{\pi\}$ per construction. Furthermore, the corners adjacent to $ \gamma_0 $ are convex.
    It also holds $\supp \chi_j \subseteq V_{I_j}$, $j=1,2,3$ as well as $\supp \chi_4 \subseteq \tilde \Omega$, so the following operator is well-defined:
    \[
    J: H^1(\Omega) \to D(q_\lambda^{I_1})\oplus D(q_\lambda^{I_2}) \oplus D(q_\lambda^{I_3}) \oplus D(n^{\tilde \Omega}_\lambda), \quad f \mapsto (\chi_1 f,\chi_2 f, \chi_3 f, \chi_4 f).
    \]
    While the inclusions of $\chi_2f \in D(q^{I_2}_\lambda)$, $\chi_3 f \in D(q^{I_3}_\lambda)$, and $\chi_4 f \in H^1(\tilde\Omega)$ are obvious, $\chi_1 f \in D(q_\lambda^I)$ needs further proof. It is apparent that
    \[
     \chi_1 f \in \{ u \in H^1(V_{I_1}) \mid  \exists \, c_0 \in I_1: u(x,y) = 0 \,\forall \, x \geq c_0 \} =:\tilde D.
    \]
    Using the density result shown in \cite[Lemma 8]{Pankrashkin_2025},  it follows that $H^1_I(V_I)$ is dense in $ \tilde D$ and therefore $\chi_1 f \in \tilde D \subseteq  D(q_\lambda^I)$. Using Lemma \ref{lem: IMS partition} we then get for some $c_0 >0$
    \[
    n_\lambda^\Omega(f) +c_0\lambda^{2\kappa} \| f\|^2_{L^2(\Omega)} \geq ( q_\lambda^{I_1} \oplus q_\lambda^{I_2} \oplus q_\lambda^{I_3} \oplus n_\lambda^{\tilde\Omega}) (Jf) ,
    \]
    and in combination with Lemma \ref{comparing operators} we arrive at
    \begin{equation}
    \label{eq: NlambdaOmega geq}
    E_n(N_\lambda^\Omega) +c_0\lambda^{2\kappa} \geq E_n( Q_\lambda^{I_1} \oplus Q_\lambda^{I_2} \oplus Q_\lambda^{I_3} \oplus N_\lambda^{\tilde\Omega}) .
    \end{equation}
     As $\tilde \Omega$ is a curvilinear polygon, \cite[Theorem 2.1]{Bonnaillie_2006} gives for some $c_1>0$ the lower bound
    \[
    N_\lambda^{\tilde \Omega} \geq c_1 \lambda.
    \]
    Furthermore, due to  Lemma \ref{lem: lower bound middle} and Lemma \ref{lem: lower bound away from corner}, we know that if $b_1$ is chosen sufficiently large, there exist $c_2,c_3>0$ such that
    \[
    Q_\lambda^{I_2} \geq c_2 \lambda^{2(1-q\kappa)} >\lambda^{\frac{2}{q+1}} ,\quad Q_\lambda^{I_3} \geq c_3 \lambda  > \lambda^{\frac{2}{q+1}}
    \]
    hold as $\lambda \to +\infty$.
    Lastly, as we get $E_n(Q_\lambda^{I_1}) \leq \lambda^{\frac{2}{q+1}} E_n(T_1) + \mathcal{O}(\lambda^{2\kappa})$ from Corollary \ref{cor: upper bound corner}, so we can use Corollary \ref{cor: lower bound corner} and the equation (\ref{eq: NlambdaOmega geq}) simplifies to
    \[
    E_n(N_\lambda^\Omega) \geq E_n(Q_\lambda^{I_1}) - c_0 \lambda^{2\kappa}
    \geq \lambda^{\frac{2}{q+1}}E_n(T_1) - c_4\lambda^{\frac{2q+3}{q+1}-3q\kappa} - c_0 \lambda^{2\kappa}
    \]
    for some $c_4>0$ which concludes our proof.
\end{proof}

    The main Theorem \ref{main theo} then follows directly from the two Propositions \ref{lem: upper bound Omega} and \ref{lem: lower bound Omega} by doing a rescaling of $T_\lambda$ that gives rise to the simplified model operator $T$ and optimizing the remainder by setting $\kappa = \frac{1}{q+1} \frac{2q+3}{2+3q} $.

\printbibliography

\end{document}